\documentclass{article}
\usepackage{amsmath,amssymb,amsfonts,amsthm}
\usepackage{mathtools,mathrsfs}
\usepackage{color}
\usepackage{hyperref}
\usepackage{cleveref}
\usepackage[a4paper]{geometry}
\usepackage{algorithm}
\usepackage{algpseudocode}
\usepackage{tikz}
\usepackage{pgfplots}
\usepackage{pgfplotstable}
\pgfplotsset{compat=1.16,width=.7\textwidth}

\newcommand{\norm}[1]{\lVert #1 \rVert}
\newcommand{\R}{\ensuremath{\mathbb{R}}}
\newcommand{\C}{\ensuremath{\mathbb{C}}}
\newcommand{\W}{\ensuremath{\mathbb{W}}}
\newcommand{\rat}{\mathcal{Q}}

\newcommand{\pol}{\ensuremath{\mathbb{P}}}

\newtheorem{lemma}{Lemma}[section]
\newtheorem{theorem}[lemma]{Theorem}
\newtheorem{corollary}[lemma]{Corollary}

\theoremstyle{remark}
\newtheorem{remark}[lemma]{Remark}

\theoremstyle{definition}
\newtheorem{definition}[lemma]{Definition}

\renewcommand{\vec}[1]{\mathbf{#1}}		

\author{
    Angelo A. Casulli\thanks{
        Scuola Normale Superiore, Pisa, Italy
        (\texttt{angelo.casulli@sns.it}).
    } 
    \and
    Leonardo Robol\thanks{
        Dipartimento di Matematica, Università di Pisa
        (\texttt{leonardo.robol@unipi.it}).
    }
}
\title{An efficient block rational Krylov solver for Sylvester equations with 
  adaptive pole selection}

\begin{document}
    \maketitle 

    \begin{abstract}
        We present an algorithm for the solution of Sylvester equations 
        with right-hand side of low rank. The method is based on projection
        onto a block rational Krylov subspace, with two key contributions 
        with respect to the state-of-the-art. First, we show how to maintain 
        the last pole equal to infinity throughout the iteration, by means 
        of pole reordering, which allows for a cheap evaluation of the 
        true residual at every step. Second, we extend the convergence 
        analysis in [Beckermann B., \emph{An error analysis for rational {G}alerkin projection applied to the
        {S}ylvester equation}, SINUM, 2011] to the block case. 
        This extension allows us to link the convergence with the problem of minimizing 
        the norm of a small rational matrix over the spectra or field-of-values 
        of the involved matrices. This is in contrast with the non-block case, 
        where the minimum problem is scalar, instead of 
        matrix-valued. Replacing the norm of the objective function 
        with a more easily evaluated
        function yields several adaptive pole selection strategies, providing 
        a theoretical analysis for known heuristics, as well as effective novel 
        techniques. 
    \end{abstract}

    \section{Introduction}

    We are concerned with the solution of Sylvester equations of the form 
    \begin{equation}\label{eqn:Sylvester}
        AX - XB = \vec u \vec v^H, \qquad \vec u \in \mathbb C^{n \times b}, \ 
        \vec v \in \mathbb C^{m \times b}, 
    \end{equation}
    and $A,B$ are square matrices of sizes $n \times n$ and $m \times m$, 
    respectively. The matrices $\vec u, \vec v$ are block vectors, i.e., 
    matrices with a few columns, with $b \ll n,m$.
    If $A$ and $B$ have disjoint spectra, the solution is unique 
    and can be expressed in the integral form
        \begin{equation}\label{thm:solution-sylv}
            X=\frac{1}{2\pi i}\int_{\gamma}(zI_n-A)^{-1}\vec u \vec v^H(zI_m-B)^{-1} dz
        \end{equation}
    where $\gamma$ is a compact contour that encloses once, 
    in positive orientation, 
    the eigenvalues of $A$, but not the eigenvalues of $B$
     \cite{lancaster1970explicit}. 
     
    Sylvester equations arise often in control theory \cite{antoulas2005approximation,benner2015survey}, and in the 
    solution of 2D PDEs on tensorized domains \cite{palitta2016matrix,townsend2015automatic}. In 
    this setting the matrices involved are often of large size, and exploiting 
    the low-rank structure in the right-hand side is essential. For problems 
    arising from control theory, the rank is linked with the number of 
    inputs and outputs in the system, so $b$ is typically moderate and related to the analysis of MIMO systems \cite{antoulas2005approximation}. 
    For PDEs, 
    the low-rank property holds in an approximate sense and is related 
    to the regularity of the problem under consideration. 
    
    When the spectra 
    of $A$ and $B$ are well-separated, one can show that the matrix $X$ that solves \eqref{eqn:Sylvester}
    has exponentially decaying singular values \cite{beckermann2017singular}, 
    and can be approximated as a low-rank matrix \cite{simoncini2016computational}.
    If $X$ is close to a low-rank matrix, i.e., 
    we can write it as $X = U Y V^H + E$ where $U,V$ 
    are matrices with a few orthogonal columns, and $E$ is a 
    small error, then the Sylvester equation 
    can be approximately solved by computing the exact solution 
    of the projected 
    equation $(U^H A U) Y - Y (V^H B V) = U^H \vec u \vec v^H V$. This is 
    the core idea of projection methods. The main difficulty is identifying 
    good bases $U,V$ to use for projection the equation. 

    A common choice is to take $U$ and $V$ as orthonormal basis of Krylov or rational Krylov subspaces. 
    When $\vec u, \vec v$ are vectors, these subspaces contain a basis 
    for $f(A) \vec u$ or $f(B^H) \vec v$, where $f(z)$ is a low degree 
    polynomial or rational function with assigned poles. Increasing the
    degree produces a sequence of subspaces, one contained in the other, 
    and therefore a sequence of approximations. The characterization 
    through polynomials and rational functions allow us to link 
    the convergence of the method with a polynomial (resp. rational) 
    approximation problem, which allows us to state explicit results (at 
    least in the case of normal matrix coefficients) 
    \cite{beckermann2011error,beckermann2009error}. 
    The rational methods are inherently 
    more complex to analyze because a choice of poles is involved, and the 
    convergence is dependent on the quality of these poles. 

    When $\vec u$ and $\vec v$ are block vectors an analogous construction 
    can be made, by building a basis for the column spans of 
    $f(A) \vec u$ or $f(B^H) \vec v$. The results in the literature 
    focus mostly on the non-block case, and are more scarce for this 
    setting. One of the contributions of this work is to extend 
    the convergence analysis for rational Krylov
    found in \cite{beckermann2011error} to this more 
    general setting. This is done by exploiting the notation
    for characteristic matrix 
    polynomial used in \cite{lund2018new} to analyze various block 
    polynomial Krylov methods. 

    If $X = UYV^H$ with $U,V$ bases of a Krylov subspace of order 
    $\ell$, then the residual 
    $AX - XB - C$ belongs to the Krylov subspace of order $\ell + 1$ \cite{simoncini2016computational}.
    This property
    can be exploited to compute the residual error almost for free 
    at each step. For 
    rational Krylov subspaces, the analogous result tells us that the residual 
    belongs to a larger subspace obtained by adding an infinity pole. However, 
    if infinity poles are periodically injected in the space, 
    we may incur in an artificial inflation of 
    the size of the projected problem. In this work, we show how one can exploit
    the theory of block rational Arnoldi decomposition (BRAD) from \cite{elsworth2020block}
    and the reordering of the poles in the subspaces to maintain a single 
    infinity pole in the definition of the rational block Krylov subspace, 
    precisely with the aim of checking the residual.

    Then, the convergence analysis introduced by extending the results 
    in \cite{beckermann2011error} is used 
    to design an adaptive-pole-selection algorithm. Since the objective 
    function is now matrix-valued, instead of scalar,  
    the problem is much richer. In particular, the minimization 
    of its norm is numerically challenging, and it 
    is natural to replace the objective function with a simpler surrogate. 
    We present various options, and we show that one of these leads to the 
    same heuristic proposed by Druskin and Simoncini in 
    \cite{druskin2011adaptive}
    generalizing the rank $1$ case. Hence, our theory provides a theoretical 
    analysis to the convergence of this choice. Then, we show that other 
    choices for the surrogate function are possible; 
    in particular, we provide an adaptive technique of pole selection
    that slightly improves the one proposed in \cite{druskin2011adaptive}. 

    The paper is structured as follows. In Section~\ref{sec:notation}
    we introduce the notation used in the paper, and then in Section~\ref{sec:matrix-poly-rat-functions}
    we discuss the tools needed from the theory of matrix polynomials and rational functions. 
    Section~\ref{sec:Rat-krylov} is devoted to the introduction of 
    rational block Krylov subspaces and the related theory, and 
    Section~\ref{sec:Krylov-Sylvester} presents the Algorithm based 
    on projection on these subspaces for the solution of Sylvester equations.
    Section~\ref{sec:residual-and-pole-selection} discusses the convergence and the 
    adaptive pole selection. Finally, we present some numerical tests
    in Section~\ref{sec:numerical}.


    \section{Notation} \label{sec:notation} Given a matrix $A$ we denote by
    $\Lambda(A)$ its spectrum, by $\W(A)$ its field of values and by $\sigma(A)$
    the set of its singular values. We use $\bar{A}$ and $A^H$ to denote the
    conjugate and the conjugate transpose of $A$, respectively. For any
    polynomial $Q(z)$ we use $\bar{Q}(z)$ to denote the polynomial that has as
    coefficients the conjugate of the coefficients of $Q(z)$. The identity
    matrix of size $s$ is denoted by $I_s$. We use bold letters to indicate
    block vectors, that is, tall and skinny matrices. The size of blocks is
    denoted by $b$.  The Frobenius norm and the two norm are denoted by
    $\norm{\cdot}_F$ and $\norm{\cdot}_2$, respectively. We employ a Matlab-like
    notation for submatrices, for instance, given $A\in \C^{m\times n}$ the
    matrix $A_{i_1:i_2,j_1:j_2}$ is the submatrix obtained selecting only rows from
    $i_1$ to $i_2$ and columns from $j_1$ to $j_2$ (extrema included). To simplify
    the notation we use bold letters also to denote block indices, that is, we use $\vec s$ to denote the set of indices $b(s-1)+1:bs$. We use the symbols
    $\otimes$  and $\oplus$ to denote the Kronecker product and the Kronecker
    sum respectively, and the symbol $\text{vec}$ to denote the operator that
    transforms a matrix into a vector obtained by stacking the columns of the
    matrix on top of one another. We
    denote by $\vec e_i$ the block vector defined as $e_i\otimes I_b$, where $e_i$ is the $i$th element of the canonical basis.

    \section{Matrix polynomials and rational functions}
    \label{sec:matrix-poly-rat-functions} In this section, we provide some
    definitions and properties about matrix polynomials that we use in the
    paper. Matrix polynomials can be equivalently interpreted as 
    polynomials with a scalar variable and matrix coefficients or 
    as a matrix with polynomial entries. Both interpretations can 
    be useful for proving different results. Formally, 
    we will denote by $\pol(\C^{b\times b})$ the space of $b \times b$ 
    matrix polynomials, with coefficients in
    $\C^{b \times b}.$ We use the notation 
    $\pol_d(\C^{b\times b})$ to denote the set of matrix
    polynomials of degree less or
    equal than $d$. A matrix polynomial is said to be monic if its leading
    coefficient is equal to the identity. 

    We will use the notation $P(z)=\sum_{i=0}^d z^i \Gamma_i$ 
    to indicate a generic matrix polynomial of degree less than 
    $d$ with matrix coefficients $\Gamma_i\in \C^{b\times b}$. 
    In order to analyze (block) Krylov methods, we associate a matrix 
    polynomials with a linear operator that acts on block 
    vectors. More precisely, we define an operator 
    $\circ$ as a function from
    $\C^{n\times n}\times\C^{n\times b}$ to $\C^{n\times b}$ as follows: 
    given
    two matrices $A\in \C^{n\times n}$ and $\vec v\in \C^{n\times b},$ we set
    \begin{equation*} 
    P(A)\circ \vec v :=\sum_{i=0}^d A^i\vec v\Gamma_i.
    \end{equation*} 
    This notation has already been used  
    in \cite{kent1989chebyshev,simoncini1996ritz,simoncini1996convergence}, 
    and has been exploited 
    in \cite{lund2018new} for the analysis of block Krylov subspaces.     
    If the matrix $A$ is fixed, the map $\vec v \mapsto P(A) \circ \vec v$ is a
    function  from $\C^{n\times b}$ to $\C^{n\times b}$. When dealing 
    with rational Krylov method, it will often be useful to 
    apply the inverse of the operator, that is given 
    a generic vector $\vec v$ finding another 
    block vector $\vec w$ such that $P(A) \circ \vec w = \vec v$. 
    Since the operator is linear in $\vec w$, this is equivalent to 
    solving a linear system. A formal definition can be given 
    as follows. 
    
    \begin{definition} \label{def:pinv}
        Given a matrix $A\in \C^{n\times n}$, a block vector $\vec v\in \C^{n\times b}$ and a matrix polynomial $P(z)=\sum_{i=0}^d z^i \Gamma_i\in\pol(\C^{b\times b})$, such that $det(P(\lambda))\neq 0$ for each $\lambda$ eigenvalue of $A$, we define $P(A)\circ^{-1}\vec v$ as the block vector $\vec w\in \C^{n\times b},$ such that $P(A)\circ \vec w= \vec v.$
    \end{definition}

    Since $\vec w$ is implicitly defined as the solution of a 
    linear system, we shall check that the system is invertible 
    to ensure that the definition is well posed. 

    \begin{lemma}
        Given a matrix $A \in \mathbb C^{n \times n}$, 
        a block vector $\vec v \in \mathbb C^{n \times b}$, and 
        a matrix polynomial $P(\lambda)$ as above such that 
        $\det(P(\lambda)) \neq 0$ for $\lambda \in \Lambda(A)$, there 
        is a unique $\vec w \in \mathbb C^{n \times b}$ 
        verifying $P(A) \circ \vec w = \vec v$. 
    \end{lemma}

    \begin{proof}
    The relation $P(A)\circ \vec w=\vec v$ can be rewritten as
    $\text{vec}(P(A)\circ \vec w)=\text{vec}(\vec v)$, in
    addition, we note that
    \begin{equation*}
        \text{vec}(P(A)\circ \vec w)=\left(\sum_{i=0}^d 
          \Gamma_i^T\otimes A^{i}\right)  \text{vec} (\vec w),
    \end{equation*}
    where $\otimes$ denotes the Kronecker product, 
    and we used the standard Kronecker 
    relation $\text{vec}(AXB) = (B^T \otimes A) \text{vec}(X)$. We now 
    prove that the matrix $\sum_{i=0}^d \Gamma_i^T\otimes A^{i}$ is invertible, 
    which implies the sought claim, since $\vec w$ can be defined as 
    \begin{equation*}
        \vec w=\text{vec}^{-1}\left(\left(\sum_{i=0}^d \Gamma_i^T\otimes A^{i}\right)^{-1}\text{vec} (\vec v)\right).
    \end{equation*}
    Let $A=UTU^H$ be a Schur decomposition of $A$, with $T$ upper
    triangular, then
    \begin{equation*}
        \sum_{i=0}^d \Gamma_i^T\otimes A^{i}=\left(I_b\otimes U\right)\left(\sum_{i=0}^d \Gamma_i^T\otimes T^{i}\right)\left(I_b\otimes U^{H}\right).
    \end{equation*}
    There exists a permutation
    matrix $P\in \C^{n b\times n b}$ (the ``perfect shuffle'', see \cite{golub2013matrix}), such that
    \begin{equation*}
        \sum_{i=0}^d \Gamma_i^T\otimes T^{i}=P\left(\sum_{i=0}^d T^{i}\otimes \Gamma_i^T\right)P^H.
    \end{equation*}
    Hence, it is sufficient to prove the invertibility of $\sum_{i=0}^d
    T^{i}\otimes \Gamma_i^T$ that is a block triangular matrix with block
    diagonal matrices given by $P(\lambda_1)^T,\dots, P(\lambda_n)^T$, where
    $\lambda_i$ are the eigenvalues of $A$. Therefore, the assumption 
    $\det(P(\lambda))\neq 0$ for each $\lambda$ eigenvalue of $A$ yields the 
    claim. 
    \end{proof}
    \begin{remark}
        The proof of well-posedness 
        of Definition~\ref{def:pinv} also gives us an explicit representation of $P(A)\circ^{-1}$: for any $\vec v\in \C^{n\times b}$
        \begin{equation*}
            P(A)\circ^{-1}\vec v=\text{vec}^{-1}\left(\left(\sum_{i=0}^d \Gamma_i^T\otimes A^{i}\right)^{-1}\text{vec} (\vec v)\right).
        \end{equation*} 
        In particular, the hypothesis $\det(P(\lambda)) \neq 0$ for $\lambda \in \Lambda(A)$ is necessary to guarantee the invertibiliy of $\sum_{i=0}^d \Gamma_i^T\otimes A^{i}$.
    \end{remark}

    The previous definitions and results essentially deal with 
    matrix polynomials; for rational Krylov methods, we will need a 
    way to incorporate rational functions into the picture. In practice, 
    it will be sufficient to consider objects of the form 
    $Q(\lambda)^{-1} P(\lambda)$, where $Q(\lambda)$ is a scalar 
    polynomial, and $P(\lambda)$ a matrix polynomial. It is immediate 
    to check that any rational matrix (i.e., a matrix with 
    rational entries) can be always written in this form.

    The following remark suggests a way to extend
    the operators
    $\circ$ and $\circ^{-1}$ to rational matrix polynomials with scalar
    denominator.

    \begin{lemma} \label{rmk:product-scalar-polynomial}
        Let $P(z)=\sum_{i=0}^d\Gamma_i z^i\in \pol_d(\C^{b\times b})$ and let $Q(z)\in \pol_k(\C)$ be a scalar polynomial. Denoting by $\tilde{P}(z)=Q(z)P(z)=\sum_{i=0}^{d+k}\Delta_iz^i$, it holds
        \begin{equation*}
            Q(A)\cdot (P(A)\circ \vec v) =\tilde P(A)\circ \vec v \quad \text{ and }\quad Q(A)^{-1}\cdot (P(A)\circ^{-1} \vec v) =\tilde P(A)\circ^{-1}\vec v, 
        \end{equation*}
        where in the second equality we assume $\det(\tilde{P}(\lambda))\neq 0$ for each $\lambda \in \Lambda(A)$.
    \end{lemma}
    \begin{proof}
        To derive the first equality it is sufficient to prove the case of
        $Q(z)=z-\alpha$ for $\alpha \in \C$, since we can factor
        $Q(z)$ as the product of linear terms. By definition of
        $\tilde P(z)$, 
        \begin{equation*}
            \tilde P(z)= (z-\alpha)P(z)= \sum_{i=0}^{d+1} (\Gamma_{i-i}-\alpha \Gamma_{i})z^i,
        \end{equation*}
        with the convention that $\Gamma_{-1}=\Gamma_{d+1}=0.$ In particular
        $\Delta_i=\Gamma_{i-1}-\alpha\Gamma_i$. Hence,
        \begin{align*}
            \tilde P(A)\circ \vec v =& \sum_{i=0}^{d+1}A^i\vec v\Delta_i
            =\sum_{i=0}^{d}A^{i+1}\vec v\Gamma_i -\alpha \sum_{i=0}^{d}A^{i}\vec v\Gamma_i\\
            =&A\cdot P(A)\circ \vec v - \alpha P(A)\circ \vec v=(A-\alpha I_n)\cdot (P(A)\circ\vec v)=Q(A)\cdot (P(A)\circ \vec v).
        \end{align*}

        For the second identity it is sufficient to prove that $\tilde P(A)\circ
        (Q(A)^{-1}\vec w ) =\vec v$, where $\vec w=P(A)\circ ^{-1}\vec v$. Using
        the first identity, 
        \begin{equation*}
            \tilde P(A)\circ (Q(A)^{-1}\vec w )= Q(A)\cdot (P(A)\circ (Q(A)^{-1}\vec w))=Q(A)\sum_{i=0}^dA^iQ(A)^{-1}\vec w\Gamma_i.
        \end{equation*}
        Since $Q(A)$ commutes with the powers of $A$, this can be reduced to
        \begin{equation*}
            \tilde P(A)\circ (Q(A)^{-1}\vec w )=P(A)\circ\vec w.
        \end{equation*}
        By definition of $\vec w$ it follows that 
        $P(A)\circ\vec w=\vec v$, that concludes the proof.
    \end{proof}

    In view of the previous result, we can extend 
    the action of a matrix polynomial 
    $P(A) \circ \vec v$ to the case of rational matrices with 
    prescribed poles.
    
    \begin{definition}
        Let $Q(z)\in \pol(\C)$ and let 
        $R(z)\in \pol(\C^{b\times b})/Q(z)$, that is there 
        exists $P(z)\in \pol(\C^{b\times b})$ such 
        that $R(z)=P(z)/Q(z)$. Given $A\in\C^{n\times n}$
        such that $Q(A)$ is invertible 
        and $\vec v\in \C^{n\times b}$, we define
        \begin{equation*}
            R(A) \circ \vec v = Q(A)^{-1} 
            \left(  P(A)\circ \vec v \right) 
            \quad \text{and}
            \quad 
            R(A) \circ^{-1} \vec v = 
            Q(A) \left(  P(A)\circ^{-1} \vec v \right).
        \end{equation*}
    \end{definition}
    The expression of a rational matrix in the form 
    $R(z) = P(z) / Q(z)$ is not unique; however the previous definition does not depend on the representation, indeed if $R(z)=P(z) / Q(z)=\tilde{P}(z)/\tilde{Q}(z)$, then $Q(z)\tilde{P}(z)=\tilde{Q}(z)P(z),$ hence by Lemma~\ref{rmk:product-scalar-polynomial}, 
    \begin{equation}\label{eqn:first_equality}
        Q(A)\cdot (\tilde{P}(A)\circ \vec v) =\tilde{Q}(A)\cdot( P(A)\circ \vec v )
    \end{equation}
    and
    \begin{equation}\label{eqn:second_equality} Q(A)^{-1}\cdot (\tilde{P}(A)\circ^{-1} \vec v) =\tilde{Q}(A)^{-1}\cdot( P(A)\circ^{-1}\vec v).
    \end{equation} 
    Multiplying both the sides of \eqref{eqn:first_equality} on the left by $Q(A)^{-1}\cdot \tilde Q(A)^{-1}$ we obtain the well-posedness of the map $\vec v \mapsto R(A) \circ \vec v$, and multiplying both sides of \eqref{eqn:second_equality} on the left by $Q(A)\tilde Q(A)$ we have the well-posedness of the map $\vec v \mapsto R(A) \circ^{-1} \vec v$. 
    
    \begin{remark}
        If the matrix $A$ is fixed, both operators    
        \begin{equation*}
            R(A) \circ: \C^{n\times d} \rightarrow \C^{n \times d} \quad \text{ and }\quad R(A) \circ^{-1}: \C^{n\times d} \rightarrow \C^{n \times d} 
        \end{equation*}        
        are linear. As in the polynomial case, 
        the latter is only defined if $R(z)$ is 
        nonsingular over all the eigenvalues of $A$. 
    \end{remark}

    
    \begin{lemma}\label{rmk:commutativity}
        If $A,B\in \C^{n\times n}$ commute, then for every rational matrix $R(z)=P(z)/Q(z),$ where  $P(z)\in\pol(\C^{b\times b})$ and $Q(z)\in \pol(\C)$,
        \begin{equation*}
            B\cdot R(A) \circ \vec v =R(A) \circ (B\vec v),
        \end{equation*}
        moreover, if $det(P(\lambda))\neq 0$ for each $\lambda \in \Lambda(A),$
        \begin{equation*} B\cdot R(A) \circ^{-1} \vec v =R(A) \circ^{-1} (B\vec v).
        \end{equation*}
    \end{lemma}
    \begin{proof} Let $P(z)=\sum_{i=1}^dz^i\Gamma_i\in \pol_d(\C^{b\times b})$ and $Q(z)\in \pol(\C)$, such that $R(z)=P(z)/Q(z)$. Then
        \begin{equation*}
            B\cdot R(A)\circ \vec v=BQ(A)^{-1}\sum_{i=1}^dA^i\vec v\Gamma_i=Q(A)^{-1}\sum_{i=1}^dA^iB\vec v\Gamma_i=R(A)\circ (B\vec v),
        \end{equation*}
        and therefore
        \begin{align*}
            &\text{vec}(B\cdot R(A)\circ^{-1} \vec v)=(I_b\otimes B)(I_b\otimes Q(A))\left(\sum_{i=1}^d\Gamma_i^T\otimes A^i\right)^{-1}\text{vec}(\vec v)\\
            =&(I_b\otimes Q(A))\left(\sum_{i=1}^d\Gamma_i^T\otimes A^i\right)^{-1} (I_b\otimes B) \text{vec}(\vec v)=\text{vec}(R(A)\circ^{-1} (B\vec v)).   
        \end{align*}
    \end{proof}
    
    Given a matrix polynomial $P(z)=\sum_{i=0}^dz^i\Gamma_i$, we denote by
    $P^H(z)$ the matrix polynomial $P^H(z): = \sum_{i=0}^dz^i\Gamma_i^H$. 
    Similarly, we denote by $\bar{P}(z)$ the matrix polynomial 
    with complex conjugate (but not transposed) coefficients. 
    Given a function
    $R(z)=P(z)/Q(z)$, we denote by $\bar{R}(z)$ and $R^H(z)$ the rational
    functions $\bar{P}(z)/\bar{Q}(z)$ and $P^H(z)/\bar{Q}(z)$, respectively. 
    
    \begin{lemma}{\label{rmk:commutativity2}}
        Given $\vec v\in\C^{n\times b}$ and $\vec w\in\C^{m\times b}$, the 
        following identities hold:
        \begin{align*}
            R(zI_n)\circ^{-1}\vec v&=\vec v(R(z))^{-1}
            &
            R(zI_n)\circ^{-1}\vec v\vec w^H&=\vec v(R^H(\bar{z}I_m)\circ^{-1}\vec w)^H.   
        \end{align*}
    \end{lemma}
    \begin{proof} Let $P(z)=\sum_{i=1}^dz^i\Gamma_i\in \pol_d(\C^{b\times b})$ and $Q(z)\in \pol(\C)$, such that $R(z)=P(z)/Q(z)$. It holds
        \begin{align*}
            \text{vec}\left(R(zI_n)\circ^{-1}\vec v\right)=&Q(z)\left(\sum_{i=0}^d\Gamma^T_i\otimes z^i I_n\right)^{-1} \text{vec}(\vec v)\\
             =&\left(\left(R^T(z)\right)^{-1}\otimes I_n\right) \text{vec}(\vec v)= \text{vec}\left(\vec{v}(R(z))^{-1}\right),
        \end{align*}
        from which follows the first equality. For the second identity notice that
        \begin{equation*}
            R(zI_n)\circ^{-1}\vec v\vec w^H=\vec v(R(z))^{-1}\vec w^H=\vec v(\vec w(R^H(\bar{z}))^{-1})^H=\vec v(R^H(\bar{z}I_m)\circ^{-1}\vec w)^H.
        \end{equation*}
    \end{proof}
    
    The following theorem is a generalization of the Cauchy integral formula
    to the action of rational matrices. 
    
    \begin{theorem}\label{thm:generalized-cauchy} Let $A\in\C^{n \times n}$,
        $\vec v\in\C^{n\times b}$ and let $\gamma$ be a compact contour that
        encloses once the eigenvalues of $A$ with positive orientation. Then,
        for any $R(z)\in \pol(\C^{b\times b})/Q(z)$, such that $\det(R(z))\neq
        0$ for each $z$ in the compact set enclosed by $\gamma$, it holds
        \begin{equation*}
            \frac{1}{2 \pi i} \int_{\gamma}R(zI_n)\circ^{-1}
            \left[ 
                (zI_n-A)^{-1}\vec v
            \right] \, dz=R(A)\circ^{-1}\vec v.
        \end{equation*}
    
        \begin{proof}
            Let $P(z)=\sum_{i=1}^dz^i\Gamma_i,$ be such that $R(z)=P(z)/Q(z)$. Then
            \begin{align*}
                &\text{vec}\left(\int_{\gamma}R(zI_n)\circ^{-1}
                  \left[ (zI_n-A)^{-1}\vec v  \right] \, dz\right)\\
                =&\left(\int_{\gamma}Q(z)\left(\sum_{i=0}^d\Gamma_i^T\otimes z^iI_n\right)^{-1}\cdot\left(I_n\otimes(zI_n-A)^{-1}\, \right)dz\right)\text{vec}(\vec v)\\
                =&\left(\int_{\gamma}Q(z)\left(\sum_{i=0}^d\Gamma_i^Tz^{i}\right)^{-1}\otimes (zI_n-A)^{-1}\, dz\right)\text{vec}(\vec v).
            \end{align*}
            For each $s,t\in \{1,\dots,b\}$, let $f_{s,t}(z)$ be the 
            function that maps $z$ in the entry in position $(s,t)$ of 
            $Q(z)\left(\sum_{i=0}^d\Gamma_i^Tz^{i}\right)^{-1}$. Since for each 
            $z$ inside the compact set bounded by $\gamma$ it holds $\det(R(z))\neq 0$, the functions $f_{s,t}(z)$, are holomorphic on such set. Then for the Cauchy integral formula, we have
            \begin{equation*}
                \frac{1}{2\pi i}\int_{\gamma}Q(z)\left(\sum_{i=0}^d\Gamma_i^Tz^{i}\right)^{-1}_{s,t}\otimes (zI_n-A)^{-1}\, dz
                =\frac{1}{2\pi i}\int_{\gamma}f_{s,t}(z)\cdot (zI_n-A)^{-1}\, dz
                =f_{s,t}(A).
            \end{equation*}
            Then, if we denote by $F\in \C^{nb\times nb}$ the block matrix for which the block in position $(s,t)$ is defined by $f_{s,t}(A)$, we have the equivalence 
            \begin{equation*}
                F=\frac{1}{2\pi i}\int_{\gamma}Q(z)\left(\sum_{i=0}^d\Gamma_i^Tz^{i}\right)^{-1}\otimes (zI_n-A)^{-1}\, dz.
            \end{equation*}
            We now claim that 
            $
                F=(I_b\otimes Q(A))\left(\sum_{i=0}^d\Gamma_i^T\otimes A^{i}\right)^{-1},
            $
            which implies the sought results, since
            \begin{align*}
                \text{vec}\left(\frac{1}{2\pi i}\int_{\gamma}R(zI_n)\circ^{-1}(zI_n-A)^{-1}\vec v \, dz\right)=F\cdot\text{vec}(\vec v)&\\
                =(I_b\otimes Q(A))\left(\sum_{i=0}^d\Gamma_i^T\otimes A^{i}\right)^{-1}\text{vec}(\vec v)&=\text{vec}\left(R(A)\circ^{-1}\vec v\right).
            \end{align*}
    
            Hence in the following we prove that $\left(\sum_{i=0}^d\Gamma_i^T\otimes A^{i}\right)\cdot F=I_{b}\otimes Q(A)$.
    
            For any $s,t\in \{1,\dots ,b\}$, let us define $g_{s,t}(z)=\left(\sum_{i=0}^d\Gamma_i^Tz^i\right)_{s,t}$. Since 
            \begin{equation*}
                \left(\sum_{i=0}^d\Gamma_i^Tz^i\right)\cdot\left[Q(z) \left(\sum_{i=0}^d\Gamma_i^Tz^i\right)^{-1}\right]=Q(z)I_b,            
            \end{equation*} it holds
            \begin{equation}\label{eqn:kron-delta}
                Q(z)\delta_{s,t}=\sum_{r=1}^bg_{s,r}(z)f_{r,t}(z),
            \end{equation}
            where $\delta_{s,t}$ denotes the Kronecker delta.
    
            To simplify the notation, for any integer $r\in\{1,\dots, b\}$, we define $ix(r)$ as the set of indices $n(r-1)+1:n r$. For any $s,t\in \{1,\dots ,b\}$ we have 
            \begin{align*}
                \left(\left(\sum_{i=0}^d\Gamma_i^T\otimes A^{i}\right)\cdot F\right)_{ix(s),ix(t)}=&\sum_{r=1}^b\left(\sum_{i=0}^d(\Gamma_i^T)_{s,r}\cdot A^{i}\right)f(A)_{r,t}=\sum_{r=1}^b g_{s,r}(A)f_{r,t}(A)=\delta_{s,t}Q(A),
            \end{align*}
            where the last equality follows from \eqref{eqn:kron-delta}.
    
        \end{proof}
    
    \end{theorem}

    Let us now recall the concept of divisibility for matrix polynomials and the
    definition of block characteristic polynomial.
    We use the term \emph{regular} to identify matrix polynomials whose determinant is not identically zero over $\mathbb C$. 
    The following results,
    including proofs of theorems, can be found in
    \cite[Section~2.5]{lund2018new} or 
    in the more classical reference \cite[Section~7.7]{gohberg2005matrix}.

    The results extend the familiar concept of Euclidean division 
    to matrix polynomials. Matrix polynomials form a 
    non-commutative ring, so we need to 
    differentiate between left and right divisors. However, the underlying 
    idea of dividing $P(z)$ by $D(z)$ is still the same: we want to 
    write $P(z)$ as a multiple of $D(z)$ plus an additional remainder 
    term, which should be of lower degree than $D(z)$. 
    
    \begin{definition}
        Let $P(z),K(z),R(z)$ and $D(z)$ be matrix polynomials, 
        where $P(z)$ has degree $d$, $D(z)$ is regular with degree 
        less than $d$, and $R(z)$ has degree less than $\deg D(z)$. 
        $K(z)$ is defined as ``left quotient'' 
        and $R(z)$ as the ``left remainder'' of $P(z)$ 
        divided by $D(z)$ if
            \begin{equation*} 
            P(z)=D(z)K(z)+R(z).
            \end{equation*} 
        If $R(z) = 0$, we say that $P(z)$ is left divisible by $D(z)$.    
    \end{definition}

    A natural question arises: given $P(z)$ and a lower degree polynomial 
    $D(z)$, can we easily check if $D(z)$ divides $P(z)$ (i.e., if the 
    remainder of the left or right division is zero)?

    For a scalar polynomial $p(\lambda)$ and 
    a linear divisor $\lambda - s$, this amounts 
    to check if $p(s) = 0$. A similar result holds for 
    matrix polynomials as well. 
    
    \begin{theorem}{\cite[Theorem~2.17]{lund2018new}}\label{thm:poly-division}
        The matrix polynomial $P(z)\in\pol(\C^{b\times b})$ is left divisible
        by $zI_b-S$, where $S\in \C^{b\times b}$ if and only if $P(S)=0$.
    \end{theorem}
    
    \begin{definition}
        Let $P(z)$ be a matrix polynomial. A matrix $S\in \C^{b\times b}$ is called a left solvent of $P(z)$ if $P(S)=0$.
    \end{definition}
    
    In the following, we omit ``left'' when referring to quotients, 
    divisibility and solvents.

    We remark that solvents are important tools in the analysis
    of matrix polynomials. 
    They  can be used to compute a part of the 
    spectrum \cite{lancaster2012hermitian}, 
    and are closely related to the solution of one-sided 
    matrix equation that arises, for instance, in 
    some Markov chains (see \cite{bini2005numerical} and the 
    references therein). 

    We now present a possible way to construct a block characteristic 
    polynomial. In the scalar case, we may think of building 
    the characteristic polynomial of a matrix $A$ by computing 
    its eigenvalues $s_1, \ldots, s_n$, and then taking the product of 
    the linear factors $p(\lambda) = (\lambda - s_1) \ldots (\lambda - s_n)$. 
    The next theorem presents the extension of this idea
    to the block case, where the eigenvalues are replaced 
    by blocks in a block diagonal matrix similar to the original one, 
    and solvents play the role of the roots. 
    
    
    \begin{definition}\label{def:charact_poly}
        Let $A\in C^{d b \times d b}$ and $\vec v\in \C^{d b\times b}.$ A block characteristic polynomial of $A$ with respect to $\vec v$ is a matrix polynomial $P(z)\in \pol_d(\C^{b\times b})$ such that 
        \begin{equation*}
            P(A)\circ \vec v=0.
        \end{equation*}
    \end{definition}
    
    \begin{theorem}\cite[Theorem~2.24]{lund2018new} \label{thm:blk-char-poly}Let $A\in C^{d b \times d b}$ and $\vec v\in \C^{d b\times b}.$ Let $P(z)$ be a monic block characteristic polynomial of $A$ with respect to $\vec v$. Assuming that there exists a block diagonal matrix 
    \begin{equation*}
        T=\begin{bmatrix}
            \Theta_1\\&\ddots\\&&\Theta_d
        \end{bmatrix},
    \end{equation*}
    with $\{\Theta_i\}_{i=1:d}\subseteq \C^{b\times b}$
        and an invertible matrix $\mathcal{U}\in \C^{db\times db}$ such that 
    \begin{equation*}
        A=\mathcal{U}T\mathcal{U}^{-1}, 
    \end{equation*}
    and letting $W=[W_1,\dots,W_d]^T=\mathcal {U}^{-1}\vec v$, with $\{W_i\}_{i=1}^d\subseteq \C^{b\times b}$, then if $W_i$ is invertible for each $i$, it holds that
    \begin{enumerate}
        \item $S_i=W_i^{-1}\Theta_iW_i$ are solvents of P(z);
        \item if $S_i-S_j$ is nonsingular for each $i\neq j$ then
        \begin{equation*}
            P(z)=(zI_b-S_1)\dots \cdot (zI_b-S_d).
        \end{equation*}
    \end{enumerate}
    \end{theorem}

\section{Block rational Krylov methods}\label{sec:Rat-krylov}
    
Given a matrix $A\in \C^{n \times n}$, a block vector $\vec v\in \C^{n\times b}$
and a sequence of poles $\boldsymbol{\xi}_{k}=\{\xi_j\}_{j=0}^{k-1}\subseteq \C\cup
\{\infty\}\setminus \Lambda(A)$ the $k$th block rational Krylov space is defined
as
\begin{equation*}  
\rat_k(A,\vec v, \boldsymbol{\xi}_k) = \left\{ R(A) \circ \vec v : R(z) = \frac{P(z)}{Q_{k}(z)}, \text{with } P(z)\in \pol_{k-1}(\C^{b\times b})\right\},
\end{equation*}
where $Q_{k}(z) = \prod_{\xi_j\in\boldsymbol{\xi}_k, \xi_j\neq \infty}(z - \xi_j)$. For
simplicity, we sometimes denote such space by $\rat_k(A,\vec v)$ omitting poles.
Note that when choosing all poles equal to $\infty$ we recover the classical
definition of block rational Krylov subspaces.

It can be proved that $\rat_k(A,\vec v)\subseteq \rat_{k+1}(A,\vec v)$. In this
work, we will assume that the block rational Krylov subspaces are always
strictly nested, that is $\rat_k(A,\vec v)\subsetneq \rat_{k+1}(A,\vec v)$ and
that the dimension of $\rat_k(A,\vec v)$ is equal to $k b$. 

An orthonormal block basis of $\rat_{k}(A,\vec v)$ (for simplicity, we will often
just say ``orthonormal basis'') is defined as a matrix $V_k=[\vec v_1,\dots,
\vec v_k]\in \C^{n\times bk}$ with orthonormal columns, such that every block
vector $\vec v\in\rat_k(A,b)$ can be written as $\vec v=\sum_{i=1}^k\vec v_i
\Gamma_i$, for $\Gamma_i\in \C^{b\times b}.$

Krylov methods require the computation of the block 
orthogonal basis 
and the corresponding projection of the matrix $A$. 
If an orthogonal basis $V_{k+1}$is known, than the projected 
matrix is given by $A_{k+1}=V_{k+1}^HAV_{k+1}$. 

The matrix $V_{k+1}$ can be computed by a block rational 
Arnoldi Algorithm\footnote{ For
simplicity we describe a version of the algorithm that does not allows poles
equal to zero. For a more complete version of the algorithm we refer to
\cite{elsworth2020block}.}
\refeq{algorithm:block-Arnoldi}, that iteratively computes the block columns of
$V_{k+1}$ and two matrices $\underline{K}_k, \underline{H}_k\in\C^{b(k+1) \times
bk}$ in block upper Hessenberg form such that 
\begin{equation}\label{eqn:rad0}
AV_{k+1}\underline{K_k}=V_{k+1}\underline{H_k}.
\end{equation}

\begin{algorithm}
    \begin{algorithmic}
	\Require{$A \in \C^{n \times n}, \vec v \in \C^{n\times b},\boldsymbol {\xi}_{k+1}= \{\xi_0, \dots, \xi_{k}\}$}
	\Ensure{ $V_{k+1}\in \C^{n\times b(k+1)},$ $\underline{H}_k,  \underline{K}_k\in \C^{b(k+1)\times bk}$}

    \State $\vec w \gets (I-A/\xi_0)^{-1}\vec v$ \Comment{with the convention $A/\infty=0$}
    \State $[\vec v_1, \sim ]\gets \text{qr}(\vec w)$  \Comment{compute the thin QR decomposition}  

	\For{$j = 1, \dots, k$}
	\State  Compute $\vec w=(I-A/\xi_{j})A\vec v_{j}$
    \For{$i = 1, \dots, j$}
    \State  $(\underline{H}_k)_{\vec i,\vec j}\gets \vec {v}_i^H\vec w$ \Comment{where $\vec i$ and  $\vec j$ are block indices}
    \State$\vec w \gets \vec w-\vec v_j(\underline{H}_k)_{\vec i,\vec j}$
    \EndFor
    \State$[\vec v_{j+1}, (\underline{H}_k)_{\vec {j+1},\vec{j}} ]\gets \text{qr}(\vec w)$  \Comment{compute the thin QR decomposition}  
    \State$(\underline{K}_k)_{\vec i,1:\vec{j+1}b}\gets (\underline{H}_k)_{\vec{i},1:\vec{j+1}b}/\xi_{j} -\vec{e}_j,$ \Comment{where $\vec{e}_j=[0,\dots,0, I_b,0]^T$ }  
	\EndFor
    \State $V_k\gets[\vec v_1,\dots,\vec v_{k+1}]$
\end{algorithmic}

	\caption{Block Rational Arnoldi} \label{algorithm:block-Arnoldi}

\end{algorithm}

Relation \eqref{eqn:rad0} completely determines the
rational Krylov subspace, and encodes all the information regarding 
poles and column span of the starting block vector. 

\begin{definition}\label{def:BRAD}[\cite{elsworth2020block}] 
    Let $A\in \C^{n\times n}$. A relation of the form
\begin{equation*}
AV_{k+1}\underline{K_k}=V_{k+1}\underline{H_k}
\end{equation*}
is called orthonormal block rational Arnoldi decomposition (BRAD), if the
following conditions are satisfied:
\begin{enumerate}
    \item $V_{k+1}\in \C^{n\times b(k+1)}$ has orthonormal columns;
    \item $\underline{K_k}$ and $\underline{H_k}$ are $b(k+1)\times bk$ block upper Hessenberg matrices such that for each $i$ either 
    $(\underline{K}_k)_{\vec{i+1},\vec i}$ or $(\underline{H}_k)_{\vec{i+1},\vec{i}}$ (or both)
    are invertible;
\item for any $i$, 
  there exist two scalars 
  $\mu_i,\nu_i \in \C$, with at least one different
from zero, such that $\mu_i
(\underline{K}_k)_{\vec{i+1},\vec{i}}=\nu_i (\underline{H}_k)_{\vec{i+1},\vec i}$;
\item the numbers $\xi_i=\mu_i/\nu_i$ above, 
  called poles of the BRAD, are outside the
  spectrum of $A$.
\end{enumerate}
\end{definition}

\begin{remark}
    The relation \eqref{eqn:rad0} produced by the block rational Arnoldi
    algorithm is a block rational Arnoldi decomposition, see \cite[Section~2]{elsworth2020block}.
\end{remark}
\begin{remark} The matrices $\underline{H_k}$ and $\underline{K_k}$ of a block
rational Arnoldi decomposition are both full rank. This follows from
\cite[Lemma~3.2]{elsworth2020block}.
\end{remark}

The following theorem relates rational Arnoldi decompositions with rational Krylov subspaces.

\begin{theorem} Let $A\in \C^{n\times n}$, $\vec v\in \C^{n\times b}$,  $\boldsymbol{\xi}_{k+1}=\{\xi_0,\dots \xi_{k}\}$ and let $\rat_{k+1}(A,\vec v)$ be the block rational Krylov subspace with poles $\boldsymbol{\xi}_{k+1}$. Let
    \begin{equation*}
        AV_{k+1}\underline{K_k}=V_{k+1}\underline{H_k}
    \end{equation*}
    be a BRAD with poles $\{\xi_1,\dots \xi_{k}\}$, such that the first block column of $V_{k+1}$ is an orthonormal basis of the space spanned by the columns of $(I-A/\xi_0)^{-1}\vec v$. Then $V_{k+1}$ is an othonormal block basis of $\rat_{k+1}(A,\vec v)$. Moreover, the matrix obtained by taking the first $b j$ columns of $V_{k+1}$ is an orthonormal block basis for $\rat_{j}(A,\vec v)$ for each $j\le k+1$. 
\end{theorem}

For the proof of the theorem and a more detailed description of block rational Arnoldi decompositions we refer to \cite{elsworth2020block}.

Let $V_k$ be the matrix obtained by taking the first $b k$ columns of $V_{k+1}$. The computation of the projected matrix $A_k=V_{k}^HAV_{k}$ by using the formula is usually expensive if the dimension of the matrix $A$ is large.
For the case of Hermitian $A$, several methods that exploit the structure of $A_k$ have been developed to avoid expensive operations for the computation, see for instance \cite{casulli2021computation, palitta2021short}. 
In the non Hermitian case, it is more difficult to exploit a structure of $A_k$. However, if the last pole of the associated BRAD is equal to infinity the projected matrix can be easily computed as $A_k=H_kK_k^{-1}$, where $K_k$ and $H_k$ are the head $kb\times kb$ principal submatrix of $\underline{K_k}$ and $\underline{H_k}$ respectively. To prove this, notice that if the last pole is equal to infinity then the last block row of $\underline{K_k}$ has to be zero, then since $\underline{K_k}$ is full rank, $K_k$ is invertible, hence multiplying both the terms of the block rational Arnoldi decomposition \eqref{eqn:rad0} on the left by $V_{k}^H$ and on the right by $K_k^{-1}$ we obtain 
$A_k=H_k K_k^{-1}$.

A technique that is often used to compute $A_k$ is to add a pole 
equal to infinity every time we want to compute a new projected matrix. 
However, this would significantly increase the size of the block rational
Krylov subspace considered by Algorithm~\ref{algorithm:block-Arnoldi}. In the next section, we describe a way to ensure that the last pole is always equal to infinity, avoiding these additional steps.
 
\subsection{Reordering poles}\label{sec:swap-poles} We propose to start the
Krylov method with $\xi_1=\infty$, then after each step transform the block
rational Arnoldi decomposition into another one that has the last two poles
swapped. Doing this procedure after each step of the block rational Krylov
method the last pole is always equal to infinity. 

This technique has been already described for the non-block case in
\cite{guttel2010rational}. In the following, we introduce a practical way to
swap the last two poles by using unitary transformations. 

Let us consider a block rational Arnoldi decomposition 

\begin{equation}\label{eqn: Rad_hat}
    A\hat V_{k+1}\underline{\hat K_k}=\hat V_{k+1}\underline{\hat H_k}
    \end{equation}
      with poles $\{\xi_1,\dots,\xi_{k-2},\infty,\xi_k\}$. By Definition~\ref{def:BRAD}, since the second last pole is equal to infinity, the submatrix
      $(\underline{\hat K_k})_{\mathbf{k}, \mathbf{k-1}}$ is equal to zero. Moreover, to
      produce a new block rational Arnoldi decomposition that has the last pole
      equal to infinity it is sufficient to annihilate the submatrix
      $(\underline{\hat K_k})_{\vec{k+1}, \vec{k}}$, keeping the block Hessenberg
      structure of the two matrices. This can be done by employing unitary
      transformations. Let
\begin{equation*}
    Q_1R_1=\begin{bmatrix}
 (\underline{\hat K_k})_{\vec k,\vec k}\\
(\underline{\hat K_k})_{\vec{k+1},\vec k}
\end{bmatrix} 
\end{equation*}
be a thin QR decomposition and let $R_2Q_2$ be an RQ decomposition\footnote{An RQ decomposition consists in writing a matrix as the product of an upper triangular matrix times a unitary matrix. It can be computed by using the same techniques involved in the computation of a QR decomposition.} for the last
block row of
\begin{equation*}
    Q_1^H\begin{bmatrix}
        (\underline{\hat H_k})_{\vec k,\vec{k-1}}&(\underline{\hat H_k})_{\vec k,\vec k}\\
        0&(\underline{\hat H_k})_{\vec{k+1},\vec k}
        \end{bmatrix} .
\end{equation*}
Then, the matrices 
\begin{equation*}
    Q_1^H\begin{bmatrix}
        0&(\underline{\hat K_k})_{\vec k,\vec k}\\
       0&(\underline{\hat K_k})_{\vec{k+1},\vec k}
       \end{bmatrix} Q_2^H
       \quad \text{ and } \quad 
       Q_1^H\begin{bmatrix}
       (\underline{\hat H_k})_{\vec k ,\vec{k-1}}&(\underline{\hat H_k})_{\vec k ,\vec k }\\
       0&(\underline{\hat H_k})_{\vec{k+1},\vec k}
       \end{bmatrix}  Q_2^H
\end{equation*}
are block upper triangular and the last block row of the first one is equal to zero.

If we let 
\begin{align*}
V_{k+1}&=\hat V_{k+1}(I_{b(k-1)}\oplus Q_1),\\
 \qquad \underline{K_k}&=(I_{b(k-1)}\oplus Q_1^H)\underline{\hat K_k}(I_{b(k-2)}\oplus Q_2^H)\\
 \underline{H_k}&=(I_{b(k-1)}\oplus Q_1^H)\underline{\hat H_k}(I_{b(k-2)}\otimes Q_2^H),
\end{align*}
where $\oplus$ denotes the Kronecker sum, the relation
\begin{equation*}
    A V_{k+1}\underline{ K_k}= V_{k+1}\underline{ H_k}    
    \end{equation*}
    is a new block rational Arnoldi decomposition that has infinity as last pole.

The computational cost of this procedure is $\mathcal{O}(kb^3)$, which is negligible with respect to the computational cost of a step of block rational Arnoldi algorithm \ref{algorithm:block-Arnoldi}.

\remark{When we transform the matrix $\hat V_{k}$ in $V_{k}$ we only perform a
linear combination between the last two block columns. For this reason the
top-left principal $b(k-1) \times b(k-1)$ submatrix of $A_k$ is equal to
$A_{k-1}$ . Hence, to compute $A_k$ it is sufficient to determine its last block
row and column. This can be done using the relation $A_k=H_kK_k^{-1}$ and so 
\begin{equation*}
A_k\vec e_k=H_kK_k^{-1}\vec e_k \quad \text{end} \quad \vec e_k^T A_k=\vec e_k^TH_kK_k^{-1}.
\end{equation*}
}

\section{Rational Krylov for Sylvester equation} \label{sec:Krylov-Sylvester}

Krylov subspace methods are one of the most popular methods for solving the Sylvester equation \eqref{eqn:Sylvester} where
$A, B$ are large size matrices and $\vec
u, \vec v$ are tall and skinny. In such a case, the
solution can be approximated by a low-rank matrix to avoid storing the complete
solution which is prohibitive for large $n$ and $m$. We refer to
\cite[Section~4.4]{simoncini2016computational} for a more complete discussion
about the topic.

The technique described in Section \ref{sec:swap-poles} can be used for the
resolution of Sylvester equations: let $U_{h+1}$ and ${V_{k+1}}$ be orthonormal
block basis for $\rat_{h+1}(A,\vec u)$ and $\rat_{k+1}(B^H,\vec v)$
respectively, generated by the block rational Arnoldi algorithm  \ref{algorithm:block-Arnoldi}and let $U_h\in
C^{n\times bh}$ and $V_k\in \C^{m\times bk}$ be the matrices obtained removing
from $U_{h+1}$ and $V_{k+1}$ the last $b$ columns. Letting $A_h=U_h^HAU_h$ and
$B_k=V_k^HBV_k$, the solution $X$ can be approximated by
$X_{h,k}=U_h\hat{X}V_k^H$, where $\hat{X}$ solves the projected equation
\begin{equation}\label{eqn:proj-sylv}
A_h\hat X-\hat X B_k=U_h^H\vec u(V_k^H\vec v)^H.
\end{equation}

For simplicity of notation in the rest of the section, we assume that $\xi_0=\infty$, that is, 
\begin{equation*}
    U_h^H\vec u= \norm{\vec u}_2\vec e_1 \quad \text{ and } \quad V_k^H\vec 
    v=\norm{\vec v}_2\vec e_1.
\end{equation*}


If $U_{h+1}$ and ${V_{k+1}}$ are determined as described in Section \ref{sec:swap-poles}, the projected matrices $A_h$ and $B_k$ can be easily computed at each step. In the following we show that this choice of poles  also allows a cheap computation of the norm of the residual matrix
\begin{equation*}
R_{h,k}=AX_{h,k}-X_{h,k}B-\vec u\vec v^H.
\end{equation*}
Since the last pole used to generate $\rat_{h+1}(A,u)$ is always equal to infinity, the columns of $AU_h$ belongs to $\rat_{h+1}(A,u)$, that is, 
\begin{equation*}
U_{h+1}U_{h+1}^HAU_h=AU_h.
\end{equation*} 
In the same way it holds 
\begin{equation*} V_k^HBV_{k+1}V_{k+1}^H=V_k^HB.
\end{equation*}
Using the last two relations, the definition of $X_{h,k}$ and that the first block columns of $U_{h+1}$ and $V_{k+1}$ are given by the orthonormalization of $\vec u$ and $\vec v$ respectively, we can rewrite the residual as 
\begin{equation*}
\begin{split}
R_{h,k}&=U_{h+1}U_{h+1}^HAU_h\hat X V_k^H-U_h\hat X V_k^HBV_{k+1}V_{k+1}^H-U_{h+1} \norm{\vec u}_2\norm{\vec v}_2\vec e_1\vec e_1^TV_{k+1}^H\\
&=U_{h+1}\left(U_{h+1}^HAU_h\hat X 
\begin{bmatrix}I_{bh}&0\end{bmatrix}-
\begin{bmatrix}I_{bk}\\0\end{bmatrix}
\hat X V_k^HBV_{k+1}-\norm{\vec u}_2\norm{\vec v}_2\vec e_1 \vec e_1^T\right)V_{k+1}^H\\
&=U_{h+1}
\begin{bmatrix}
U_{h}^HAU_h\hat X 
-
\hat X V_k^HBV_{k}-\norm{\vec u}_2\norm{\vec v}_2\vec e_1\vec e_1^T
&-\hat X V_k^HB \vec v_{k+1}\\
\vec u_{h+1}^HAU_h\hat X&0
\end{bmatrix}
V_{k+1}^H\\
&=U_{h+1}
\begin{bmatrix}
A_h\hat X 
-
\hat X B_k-\norm{\vec u}_2\norm{\vec v}_2\vec e_1\vec e_1^T
&-\hat X V_k^HB \vec v_{k+1}\\
\vec u_{h+1}^HAU_h\hat X&0
\end{bmatrix}
V_{k+1}^H\\
&=U_{h+1}
\begin{bmatrix}
0
&-\hat X V_k^HB \vec v_{k+1}\\
\vec u_{h+1}^HAU_h\hat X&0
\end{bmatrix}
V_{k+1}^H\\
\end{split}
\end{equation*}
where $\vec u_{h+1}$ and $\vec v_{k+1}$ are the last block columns of $U_{h+1}$
and $V_{k+1}$ respectively, and the zero matrix in the top left corner of the
block matrix in the last row is given by equation \eqref{eqn:proj-sylv}.

Since the columns of $U_{h+1}$ and $V_{k+1}$ are orthonormal, the norm of the residual is equal to the norm of the block matrix
\begin{equation} \label{eqn:residual_struct}
\begin{bmatrix}
0
&-\hat X V_k^HB \vec v_{k+1}\\
\vec u_{h+1}^HAU_h\hat X&0
\end{bmatrix}.
\end{equation} 

Let us now consider the block rational Arnoldi decomposition 
\begin{equation*}
AU_{h+1}\underline{K_h}^{(A)}=U_{h+1}\underline{H_h}^{(A)}. \end{equation*}
Multiplying both the terms of the equations on the right by $\left({K_h}^{(A)}\right)^ {-1}$, where ${K_h}^{(A)}$ is the $bh\times bh$ head principal submatrix of $\underline{K_h}^{(A)}$, noting that the last block row of $\underline{K_h}^{(A)}$ is equal to zero, we have
\begin{equation}\label{eqn:smallA}
AU_{h}=U_{h+1}\underline{H_h}^{(A)}\left({K_h}^{(A)}\right)^ {-1}.
\end{equation}

Analogously, if 
\begin{equation*}
B^HV_{k+1}\underline{K_k}^{(B)}=V_{k+1}\underline{H_k}^{(B)}
\end{equation*}
is a block rational Arnoldi decomposition, we have that 
\begin{equation} \label{eqn:smallB}
B^HV_{k}=V_{k+1}\underline{H_k}^{(B)}\left({K_k}^{(B)}\right)^ {-1},
\end{equation}
where $K_k^{(B)}$ is the head $bk\times bk$ principal submatrix of $\underline{K_k}^{(B)}$.

Using the equations \eqref{eqn:smallA} and \eqref{eqn:smallB}, we can rewrite the matrix \eqref{eqn:residual_struct} as

\begin{equation} \label{eqn:small_residual}
\begin{bmatrix}
0
&-\hat X \left(K_k^{(B)}\right)^{-H}\left(\underline{H_k}^{(B)}\right)^HV_{k+1}^H \vec v_{k+1}\\
\vec u_{h+1}^HU_{h+1} \underline{H_h}^{(A)}\left(K_h^{(A)}\right)^{-1}\hat X&0
\end{bmatrix},
\end{equation}
exploiting the orthogonality of the columns of $U_{k+1}$ and $V_{k+1}$, the matrix \eqref{eqn:small_residual} is equal to

\begin{equation*}
\begin{bmatrix}
0
&\hat X \left(K_k^{(B)}\right)^{-H}\left(\underline{H_k}^{(B)}\right)^H\vec e_{k+1}^H\\
\vec e_{h+1} \underline{H_h}^{(A)}\left(K_h^{(A)}\right)^{-1}\hat X&0
\end{bmatrix},
\end{equation*}
where $\vec e_{h+1}\in \C^{b(h+1)\times b}$ and $\vec e_{k+1}\C^{b(k+1)\times b}$.

The norm of this matrix can be recovered by 
the norms of the block vectors 
\begin{equation*}
    \vec e_{h+1}\underline{H_h}^{(A)}\left(K_h^{(A)}\right)^{-1}\hat X \quad \text{ and } \quad \hat X \left(K_k^{(B)}\right)^{-H}\left(\underline{H_k}^{(B)}\right)^H \vec e_{k+1}^H.    
\end{equation*} In particular the computation of the norm of the residual does not involve the matrices $A$ and $B$, hence it can be performed with a computational cost that does not depend on $n$ and $m$.

\section {Residual and pole selection} \label{sec:residual-and-pole-selection}
The aim of this section is to prove the following theorem:
\begin{theorem} \label{thm:block-residual}
    Let $A \in \C^{n\times n}$, $B\in \C^{m\times m}$, $\vec u\in \C^{n\times b}$ and $\vec v\in \C^{m\times b}$. Let $U\in \C^{n\times b h}$ and $V\in \C^{m\times b k}$ be orthonormal block basis for $\rat_h(A,\vec u,\boldsymbol{\xi}_{h}^{(A)})$ and  $\rat_k(B^H,\vec v,\boldsymbol{\xi}_{k}^{(B)})$, respectively, and let $A_h=UAU^H$, $B_k=VBV^H$. Let $X_{h,k}=UY_{h,k}V^H$ where $Y_{h,k}$ is the solution of the Sylvester equation
    \begin{equation*}
        A_hY_{h,k}-Y_{h,k}B_k=\vec u^{(h)} (\vec v^{(k)})^H, 
    \end{equation*}
    with $\vec u^{(h)}=U^H\vec u,$ and $\vec v^{(k)}=V^H\vec v.$ 
    Let $\chi_A(z)\in \pol_h(\C^{b\times b})$ and $\chi_B(z)\in \pol_k(\C^{b\times b})$, be monic block characteristic polynomials of $A_h$ with respect to $\vec u^{(h)}$ and $B_k$ with respect to $\vec v^{(k)}$, respectively. Define 
    \begin{equation*}
        R_A^G(z)=\frac{\chi_A(z)}{Q_A(z)} \quad   \text{ and } \quad R_B^G(z)=\frac{\chi_B(z)}{Q_B(z)},
    \end{equation*}
    where 
    \begin{equation*}
        Q_A(z)=\prod_{\xi\in\boldsymbol{\xi}^{(A)}, \xi \neq \infty}(z-\xi) \quad \text{ and } \quad Q_B(z)=\prod_{\xi\in\boldsymbol{\xi}^{(B)}, \xi \neq \infty}(z-\xi).
    \end{equation*}

Then the residual matrix can be written as $R_{h,k}=\rho_{1,2}+\rho_{2,1}+\rho_{2,2}$, where
\begin{align*}
    &\rho_{1,2}=U({R_B^G}^H(A_h)\circ^{-1}\vec u^{(h)})(R_B^G(B^H)\circ \vec v)^H,\\
    &\rho_{2,1}=\left(R_A^G(A)\circ\vec u\right)({R_A^G}^H(B_k)\circ^{-1} \vec v^{(k)})^HV^H,\\
    &\rho_{2,2}=\left(R_A^G(A)\circ \vec u\left(R_A^G(\infty)\right)^{-1}\right) \left(R_B^G(B^H)\circ \vec v\left(R_B^G(\infty)\right)^{-1}\right)^H,
\end{align*}
with 
\begin{equation*}
    R_A^G(\infty)=\lim_{|\lambda|\rightarrow\infty}R_A^G(\lambda) \quad \text{ and } \quad R_B^G(\infty)=\lim_{|\lambda|\rightarrow\infty}R_B^G(\lambda).
\end{equation*}
Moreover 
\begin{equation}\label{eqn:redidual-norm-sum}
    \norm{R_{h,k}}^2_F=\norm{\rho_{1,2}}^2_F+\norm{\rho_{2,1}}^2_F+\norm{\rho_{2,2}}^2_F.
\end{equation}

\end{theorem}

\begin{remark}\label{rmk:infinity-pole}
    If one of the poles of $\vec \xi_A$ or $\vec \xi_B$ is chosen equal to infinity, then $\rho_{2,2}=0$.
\end{remark}

The representation of the residual matrix given by Theorem~\ref{thm:block-residual} allows us to provide adaptive techniques for the pole selection for the resolution of Sylvester equations.

\subsection{Proof of Theorem~\ref{thm:block-residual}}
Theorem~\ref{thm:block-residual} and the proof we provide in this section, are generalizations of the ones provided by Beckermann in \cite{beckermann2011error} for the case of classical rational Krylov methods.

Let us start by introducing some lemma that is needed for the proof of the theorem.

\begin{lemma}[Block exactness]{\label{lem:exactness}} For any $R_A(z)\in \pol_{h}(\C^{b\times b})/Q_A(z),$ we have  
    \begin{equation*}
        UU^HR_A(A)\circ\vec u=UR_A(A_h)\circ \vec u^{(h)},
    \end{equation*}
    in particular, if $R_A(z)\in \pol_{h-1}(\C^{b\times b})/Q_A(z),$ it holds
    \begin{equation*}
        R(A)\circ\vec u=UR_A(A_h)\circ \vec u^{(h)}.
    \end{equation*}     
    
    Similarly for any $R_B\in \pol_{k}(\C^{b\times b})/Q_B(z),$ we have that $VV^HR_B(B^H)\circ \vec v=VR_B(B_k)\circ \vec v^{(k)}$ and for any $R_B\in \pol_{k-1}(\C^{b\times b})/Q_B(z),$ it holds $R_B(B^H)\circ \vec v=VR_B(B_k)\circ \vec v^{(k)}$.
   
\end{lemma}
\begin{proof}
    We only prove the first two identities, since the other claims follow using
    the same argument. The proof is composed of two parts. First, we suppose
    that the poles are all equal to infinity, i.e., $Q_A(z)=1$. Then, we extend
    the proof for a generic choice of poles.

 For the first part, by linearity, it is sufficient to prove the equalities for
 $R_A(z)=z^j$ with $j\le h$. We proceed by induction on $j$. If $j=0$ there is
 nothing to prove. For $R_A(z)=z^{j+1}$, by the inductive hypothesis we have
 \begin{equation*}
    UU^HA^{j+1}\vec u=UU^HAA^j\vec{u}=UU^HAUA_h^j\vec{u}^{(h)}=UA_h^{j+1}\vec{u}^{(h)}.
 \end{equation*}
 Moreover, if $j+1\le h-1$, $A^{j+1}\vec u\in \rat_h(A,\vec{u},\vec{\infty})$,
 hence  $UU^HA^{j+1}\vec u=A^{j+1}\vec u$.

 Let now $R_A(z)=P(z)/Q_A(z)$ with $P(z)\in \pol(\C^{b\times b})$.
 Using the
 commutativity property of Lemma~\ref{rmk:commutativity}, we have that
 \begin{equation*}
    R_A(A)\circ \vec u= Q_A(A)^{-1}P(A)\circ \vec u=P(A)\circ (Q_A(A)^{-1}\vec u).
 \end{equation*}
Hence, if we let $\vec c=Q_A(A)^{-1}\vec u$, from the result of the first step we have
\begin{equation*}
    UU^HR_A(A)\circ \vec u =UU^HP(A)\circ \vec c=UP(A_h)\circ (U^H\vec c),
\end{equation*}
and, if $R_A(A) \in \pol_{h-1}(\C^{b\times b})/Q_A(A)$, we have
\begin{equation*}
    R_A(A)\circ \vec u =P(A)\circ \vec c=UP(A_h)\circ (U^H\vec c).
\end{equation*}
To conclude it is sufficient to prove that $U^H\vec c=Q(A_h)^{-1}\vec u^{(h)}$.
Since $\vec u= Q(A)\circ\vec c$, by the first step of the proof we have
\begin{equation*}
    \\UU^H\vec u=UU^HQ(A)\circ \vec c=UQ(A_h)\circ (U^H\vec c)=UQ(A_h) U^H\vec c.
\end{equation*}
Since $U^HU=I_{bh}$, multiplying both sides on the left by $Q(A_h)^{-1}U^H$ we
get
\begin{equation*}
    Q(A_h)^{-1}U^H\vec u=U^H\vec c,
\end{equation*}
that concludes the proof.
\end{proof}

\begin{corollary}\label{cor:char-poly}
    Let $\chi_A(z)\in\pol_h(\C^{b\times b})$ and $\chi_{B}(z)\in\pol_k(\C^{b\times b})$ be  monic block characteristic polynomial for $A_h$ with respect to $\vec u^{(h)}$ and $B_k$ with respect to $\vec v^{(k)}$, respectively. Let $R_A^G(z)=\chi_A(z)/Q_A(z)$ and $R_B^G(z)=\chi_B(z)/Q_B(z)$. It holds
    \begin{equation*}
        U^HR_A^G(A)\circ \vec u=0 \quad \text{ and }\quad V^HR_B^G(B)\circ \vec v=0.
    \end{equation*}
    Moreover $R_A^G(A)\circ \vec u$ minimizes
    $\norm{R(A)\circ \vec u}_F$ over all the $R(z)\in \pol_{h}(\C^{b\times b})/Q_A(z)$ such that $R(z)=P(z)/Q_A(z)$ where $P(z)$ is a monic matrix polynomial. Analogously $R_B^G(B)\circ \vec v$ minimizes
    $\norm{R(B)\circ \vec v}_F$ over all the $R(z)\in \pol_{k}(\C^{b\times b})/Q_B(z)$ with monic numerator.
\end{corollary}
\begin{proof}
    In the following, we prove the corollary for $R_A^{G}(A)\circ \vec u$. The proof for $R_B^{G}(B)\circ \vec v$ is the same. 
    By Lemma~\ref{lem:exactness} it holds
    \begin{equation*}
        UU^HR_A^G(A)\circ \vec u=UR_A^G(A_h)\circ \vec{u}^{(h)}=0.
    \end{equation*}
    Since $U^HU=I_{bh}$, multiplying on the left by $U^H$ we obtain the first equivalence.

    The problem of minimizing $\norm{R(A)\circ \vec u}_F$ over all the $R(z)\in \pol_{h}(\C^{b\times b})/ Q_A(z)$ with monic numerator can be rewritten as
    \begin{align*}
    &\min_{\hat R(z)\in \pol_{h-1}(z)/Q_A(z)}\norm{Q_A(A)^{-1}A^h\vec u - \hat R(z)\circ \vec u}_F\\
    =&\min_{\vec y\in \C^{h\times b}}\norm{Q_A(A)^{-1}A^h\vec u-U\vec y}_F\\
    =&\min_{\vec y\in \C^{h\times b}}\norm{(I_b\otimes Q_A(A)^{-1}A^h)\text{vec}(\vec u)-(I_b\otimes U)\text{vec}(\vec y)}_2.
    \end{align*}
    The solution of the least square problem is given by the matrix $\vec y$ such that 
    \begin{equation*}
        (I_b\otimes U)^H\left((I_b\otimes Q_A(A)^{-1}A^h)\text{vec}(\vec u)-(I_b\otimes U)\text{vec}(\vec y)\right)=0,      
    \end{equation*}
    that is analogue to ask that $U^H(Q_A(A)^{-1}A^h\vec u-U\vec y)=0$, that is, the solution of the minimization problem sathisfies $U^H(R(A)\circ \vec u)=0$, hence the function $R_A^G(z)$ is the solution.
\end{proof}

Lemma \ref{lem:exactness} is usually referred as the exactness property of
rational Krylov spaces. The proof is a generalization of the ones for non-block
rational Krylov methods, which is described in
\cite[Lemma~4.6]{guttel2010rational}.

\begin{lemma} \label{lem:second}
    Let $R_A(z)\in \pol_h(\C^{b\times b})/Q_A(z)$, and $z$ such that $det(R_A(z))\neq 0$. Then, 
    \begin{equation}\label{eqn:thesis-second-lemma}
        R_A(zI_n)\circ^{-1}\left[
            R_A(zI_n) \circ \vec x - R_A(A) \circ \vec x
        \right] = UR_A(zI_{bh})\circ^{-1} \left[
            R_A(zI_{bh}) \circ \tilde{\vec x} - R_A(A_h) \circ \tilde{\vec x}
        \right], 
    \end{equation}
    where $\vec x := (zI_n - A)^{-1} \vec u$ and 
    $\tilde{\vec x} := (zI_{bh} - A_h)^{-1} \vec u^{(h)}$.    
    
    Similarly, for any $R_B(z)\in \pol_k(\C^{b\times b})/Q_B(z)$ and for each $z$ such that $det(R_B(z))\neq 0$
    \begin{equation*}
        R_B(zI_m)\circ^{-1}\left[
            R_B(zI_m) \circ \vec y - R_B(B^H) \circ \vec y
        \right] = VR_B(zI_{bk})\circ^{-1} \left[
            R_B(zI_{bk}) \circ \tilde{\vec y} - R_B(B_k) \circ \tilde{\vec y}
        \right], 
    \end{equation*}
    where $\vec y := (zI_m - B^H)^{-1} \vec v$ and 
    $\tilde{\vec y} := (zI_{bk} - B_k)^{-1} \vec v^{(k)}$.
\end{lemma}

\begin{proof}
    We only derive the first equality, the second follows by an analogous argument. 
    Note that $R_A(zI_n)\circ^{-1}$ is well-defined since the fact that
    $\det(R_A(z))\neq 0$. By Lemma~\ref{rmk:commutativity2} equation
    \eqref{eqn:thesis-second-lemma} is equivalent to
    \begin{equation*}
       [ R_A(zI_n) \circ \vec x - R_A(A) \circ \vec x](R(z))^{-1}
   =
        U\left[R_A(zI_{bh}) \circ \tilde{\vec x} - R_A(A_h) \circ \tilde{\vec x}\right](R(z))^{-1},
\end{equation*}
hence, multiplying both sides on the right by $R(z)$, it is sufficient to prove
     \begin{equation*}
            R_A(zI_n) \circ \vec x - R_A(A) \circ \vec x
       =
            U\left[R_A(zI_{bh}) \circ \tilde{\vec x} - R_A(A_h) \circ \tilde{\vec x}\right]
       .
    \end{equation*}
    Since $A$ and  $(zI_n - A)^{-1}$ commute and analogously for 
    $(zI_{bh} - A_h)^{-1}$ and $A_{h}$, by Lemma~\ref{rmk:commutativity} the claim can be equivalently restated as follows:
    \begin{equation}\label{eqn:first-step}
        (zI_n - A)^{-1} \left[
            R_A(zI_n) \circ \vec u - R_A(A) \circ \vec u
        \right] = U (zI_{bh} - A_h )^{-1} \left[
            R_A(zI_{bh}) \circ \vec u^{(h)} - R_A(A_h) \circ \vec u^{(h)}
        \right].   
    \end{equation}
    To prove it, we introduce the auxiliary function $G_z(x) := R_A(z) -
    R_A(x)$. We consider $G_z(x)$ as a function in the variable $x$, and assume
    that $z$ is fixed; in particular $G_z(x)=P_z(x)/Q_A(x)$, where $P_z(x)$ is a
    matrix polynomial of degree $h$ in the variable $x$. Note that  $G_z(A)\circ
    \vec u= R_A(zI_n)\circ \vec u-R_A(A)\circ \vec u$; indeed, letting
    $P_z(x)=\sum_{i=0}^h \Delta_ix^i\in \pol_{h}(\C^{b\times b})$ and
    $Q_A(x)=\sum_{i=0}^hq_ix^i$, from the definition of $G_z(x)$ we have that
    \begin{equation*}
        R_A(x)=R_A(z)-G_z(x)=(Q_A(x)R_A(z)-P_z(x))/Q_A(x)=\left[\sum_{i=0}^h (q_iR_A(z)-\Delta_i)x^i\right]/Q_A(x).
    \end{equation*}
    Hence,
    \begin{align*}
        R_A(A)\circ \vec u &= Q_A(A)^{-1} \left[\sum_{i=0}^h A^i\vec u (q_iR_A(z)-\Delta_i)\right]\\
        &= Q_A(A)^{-1} \left[R_A(z)\sum_{i=0}^h q_iA^i\vec u \right]-Q_A(A)^{-1} \left[\sum_{i=0}^h A^i\vec u\Delta_i\right]\\
        &= R_A(z)Q_A(A)^{-1}Q_A(A) \vec u - G_z(A)\circ \vec u 
        =R_A(zI_n)\circ \vec u -G_z(A)\circ \vec u.
    \end{align*}
    Analogously, it can be proven that $G_z(A_h)\circ \vec u^{(h)}=
    R_A(zI_{bh})\circ \vec u^{(h)}-R_A(A_h)\circ \vec u^{(h)}$. Using the
    equivalences introduced before, we may rewrite \eqref{eqn:first-step} as
\begin{equation}\label{eqn:first-step-bis}
    (zI_n-A)^{-1}G_z(A)\circ \vec u= U(zI_{bh}-A_h)^{-1}G_z(A_h)\circ\vec u^{(h)}.
\end{equation}
By definition, evaluating $G_z(x)$ 
    at $x = zI_b$ yields 
    $
        G_z(zI_b) = R_A(z) - R_A(zI_b) = 0. 
    $
    This implies that the linear matrix polynomial $(xI_b - zI_b)$ is a 
    left solvent for $P_z(x)$, and we may write 
    \begin{equation*}
        \tilde G_z(x) := 
        (z - x)^{-1} G_z(x) = 
        -(xI_b - zI_b)^{-1} G_z(x) \in \pol_{h-1}(\C^{b\times b})/Q_A(x). 
\end{equation*}    
    Thanks to the exactness from Lemma~\ref{lem:exactness} we have $       
    \tilde G_z(A) \circ \vec u =
        U  \tilde G_z(A_h) \circ \vec u^{(h)},
    $
    that by Lemma~\ref{rmk:product-scalar-polynomial} is equal to \eqref{eqn:first-step-bis}, concluding the proof.
\end{proof}

\begin{lemma}\label{lem:third}
    Let $\chi_A(z)\in\pol_h(\C^{b\times b})$ and $\chi_{B}(z)\in\pol_k(\C^{b\times b})$ be block characteristic polynomial for $A_h$ with respect to $\vec u^{(h)}$ and $B_k$ with respect to $\vec v^{(k)}$, respectively. Let $R_A^G(z)=\chi_A(z)/Q_A(z)$ and $R_B^G(z)=\chi_B(z)/Q_B(z)$. We have that
    \begin{equation*}
        (zI_n-A)^{-1}\vec u-U(zI_{bh}-A_h)^{-1}\vec u^{(h)}=R_A^G(zI_n)\circ^{-1}R_A^G(A)\circ(zI_n-A)^{-1}\vec u,
    \end{equation*}
    and
    \begin{equation*}
        (zI_m-B^H)^{-1}\vec v-U(zI_{bk}-B_k)^{-1}\vec v^{(k)}=R_B^G(zI_m)\circ^{-1}R_B^G(B^H)\circ(zI_m-B^H)^{-1}\vec v,
    \end{equation*}
\end{lemma}
\begin{proof}
    It follows from Lemma~\ref{lem:second} observing that $R_A^G(A_h)\vec u^{(h)}=0$ and $R_B^G(B_k)\vec v^{(k)}=0$.
\end{proof}

We are now ready to give the proof of Theorem~\ref{thm:block-residual}:
\begin{proof}[Proof of Theorem~\ref{thm:block-residual}]
    To simplify the notation we define $\vec x=(zI_n-A)^{-1}\vec u$, $\tilde {\vec x} =(zI_{bh}-A_h)^{-1}\vec u^{(h)}$, $\vec y=(\bar{z}I_m-B^H)^{-1}\vec v$ and  $\tilde {\vec y} =(\bar{z}I_{bk}-B_k)^{-1}\vec v^{(k)}$. According to 
    Equation~\eqref{thm:solution-sylv}, letting $X$ the solution of the Sylvester equation, we have
    \begin{equation*}
        X-X_{h,k}=\frac{1}{2\pi i}\int_{\gamma_A}\vec x \vec y^H-U\tilde{\vec x} \tilde{\vec y}^HV^H dz,
    \end{equation*}
    where $\gamma_A$ is a compact contour with positive orientation
    that encloses 
    the eigenvalues of $A$ and $A_h$, but not the
    eigenvalues of $B$ and $B_k$. Using Lemma~\ref{lem:third} we have
    \begin{align}
        X-X_{h,k}&=\frac{1}{2\pi i}\int_{\gamma_A}\left((\vec x-U\tilde{\vec x})\vec y^H+\vec x(\vec y-V\tilde{\vec y})^H-(\vec x-U\tilde{\vec x})(\vec y-V\tilde{\vec y})^H\right)dz\\
        &\label{enq:integral1}=\frac{1}{2\pi i} \int_{\gamma_A} \left(R_A^G(zI_n)\circ^{-1}R_A^G(A)\circ\vec x\right) \vec{y}^H dz\\
        &\label{enq:integral2}+\frac{1}{2\pi i} \int_{\gamma_A}\vec x\left(R_B^G(\bar{z}I_m)\circ^{-1}R_B^G(B^H)\circ\vec y\right)^H dz\\
        &\label{enq:integral3}-\frac{1}{2\pi i} \int_{\gamma_A} \left(R_A^G(zI_n)\circ^{-1}R_A^G(A)\circ\vec x\right)\left(R_B^G(\bar{z}I_m)\circ^{-1}R_B^G(B^H)\circ\vec y\right)^H dz.
    \end{align}
    The residual matrix can be written as $R_{h,k}=A(X-X_{h,k})-(X-X_{h,k})B$
    that is the sum of the three differences of integrals
    $A\mathcal{S}-\mathcal{S}B$, where $\mathcal S$ is substituted by
    \eqref{enq:integral1}, \eqref{enq:integral2} and \eqref{enq:integral3}. In
    the following, we study each difference of integrals separately.
    Concerning~\eqref{enq:integral1}, by Lemma~\ref{rmk:commutativity} we have

    \begin{align}
        & {\phantom{=}} \frac{1}{2\pi i}A \int_{\gamma_A} \left(R_A^G(zI_n)\circ^{-1}R_A^G(A)\circ\vec x\right) \vec{y}^H dz- 
            \frac{1}{2\pi i} \int_{\gamma_A} \left(R_A^G(zI_n)\circ^{-1}R_A^G(A)\circ\vec x\right) \vec{y}^HB dz\\
        &= \label{eqn:couple-integral1}\frac{1}{2\pi i} \int_{\gamma_A} \left(R_A^G(zI_n)\circ^{-1}R_A^G(A)\circ A\vec x\right) \vec{y}^H dz-   
            \frac{1}{2\pi i} \int_{\gamma_A} \left(R_A^G(zI_n)\circ^{-1}R_A^G(A)\circ\vec x\right) (B^H\vec{y})^H dz.
    \end{align}
    Let now $\gamma_B$ be a positively oriented compact contour that encloses
    the eigenvalues of $B$ and $B_k$, but not the eigenvalues of $A$ and $A_h$.
    Since the integrand is $\mathcal{O}(z^{-2})_{z\rightarrow \infty}$, we can
    replace $\gamma_A$ with $\gamma_B$ just by changing the sign of the
    integral.
    Noting that 
    \begin{equation} \label{eqn:relations-Ax-By}
    A\vec x=(A-zI_n)\vec x +zI_n\vec x=-\vec u +z \vec x  \quad \text{ and analogously,  } \quad B^H\vec y=-\vec v+\bar{z}\vec y,
    \end{equation}
    the sum of integrals in \eqref{eqn:couple-integral1} can be rewritten as 
    \begin{equation*}
        -\frac{1}{2\pi i} \int_{\gamma_A} \left(R_A^G(zI_n)\circ^{-1}R_A^G(A)\circ \vec u\right) \vec{y}^H dz+  
            \frac{1}{2\pi i} \int_{\gamma_A} \left(R_A^G(zI_n)\circ^{-1}R_A^G(A)\circ\vec x\right) \vec{v}^H dz.
    \end{equation*}
    Then, changing $\gamma_A$ with $\gamma_B$ we obtain 
    \begin{equation}\label{eqn:first-residual}
        \frac{1}{2\pi i} \int_{\gamma_B} \left(R_A^G(zI_n)\circ^{-1}R_A^G(A)\circ \vec u\right) \vec{y}^H dz,
    \end{equation}
    since the integral
    \begin{equation*}        
            \frac{1}{2\pi i} \int_{\gamma_B} \left(R_A^G(zI_n)\circ^{-1}R_A^G(A)\circ\vec x\right) \vec{v}^H dz ,           
    \end{equation*}
    vanishes for the residual theorem.

    The same technique can be used to write the second difference of integrals as
    \begin{equation}\label{eqn:second-residual}
        \frac{1}{2\pi i} \int_{\gamma_A} \vec x\left(R_B^G(\bar{z}I_m)\circ^{-1}R_B^G(B^H)\circ \vec v\right)^Hdz.
    \end{equation}
    
     Using again the relations in \eqref{eqn:relations-Ax-By}, the third difference of integrals can be written as $ I_{3,1}+I_{3,2}$, where

     \begin{equation*}
        I_{3,1}=\frac{1}{2\pi i} \int_{\gamma_A} \left(R_A^G(zI_n)\circ^{-1}R_A^G(A)\circ \vec u\right) \left(R_B^G(\bar{z}I_m)\circ^{-1}R_B^G(B^H)\circ \vec y\right)^H dz,
     \end{equation*}
     and
     \begin{equation}\label{eqn:third-residual}
        I_{3,2}= -\frac{1}{2\pi i} \int_{\gamma_A} \left(R_A^G(zI_n)\circ^{-1}R_A^G(A)\circ \vec x\right) \left(R_B^G(\bar{z}I_m)\circ^{-1}R_B^G(B^H)\circ \vec v\right)^H dz.
    \end{equation}
    For a generic choice of poles, it is only guaranteed that the integrand of
    $I_{3,1}$ is $\mathcal{O}(z^{-1})_{z\rightarrow \infty}$ hence, changing
    $\gamma_A$ with $\gamma_B$, we can rewrite $I_{3,1}$ as 
    \begin{align}
                &\label{eqn:fourth-residual}\left(R_A^G(\infty \cdot I_n)\circ^{-1}R_A^G(A)\circ \vec u\right) \left(R_B^G(\infty\cdot I_m)\circ^{-1}R_B^G(B^H)\circ \vec v\right)^H\\
                &\label{eqn:fifth-residual}-\frac{1}{2\pi i} \int_{\gamma_B} \left(R_A^G(zI_n)\circ^{-1}R_A^G(A)\circ \vec u\right) \left(R_B^G(\bar{z}I_m)\circ^{-1}R_B^G(B^H)\circ \vec y\right)^H dz.
    \end{align}    
    Summing \eqref{eqn:first-residual}, \eqref{eqn:second-residual}, \eqref{eqn:third-residual}, \eqref{eqn:fourth-residual} and \eqref{eqn:fifth-residual}, we obtain
    \begin{align*}
        R_{h,k}&=\left(R_A^G(\infty \cdot I_n)\circ^{-1}R_A^G(A)\circ \vec u\right) \left(R_B^G(\infty\cdot I_m)\circ^{-1}R_B^G(B^H)\circ \vec v\right)^H\\
        &+\frac{1}{2\pi i} \int_{\gamma_B} \left(R_A^G(zI_n)\circ^{-1}R_A^G(A)\circ \vec u\right) \left( \left(I_m-R_B^G(\bar{z}I_m)\circ^{-1}R_B^G(B^H)\right)\circ \vec y\right)^H dz\\
        &+\frac{1}{2\pi i} \int_{\gamma_A} \left(\left(I_n-R_A^G(zI_n)\circ^{-1}R_A^G(A)\right)\circ \vec x\right) \left(R_B^G(\bar{z}I_m)\circ^{-1}R_B^G(B^H)\circ \vec v\right)^H dz.
    \end{align*}
    Applying Lemma \ref{lem:third}, we have 
    \begin{align*}
        R_{h,k}&=\left(R_A^G(\infty \cdot I_n)\circ^{-1}R_A^G(A)\circ \vec u\right) \left(R_B^G(\infty\cdot I_m)\circ^{-1}R_B^G(B^H)\circ \vec v\right)^H\\
        &+\frac{1}{2\pi i} \int_{\gamma_B} \left(R_A^G(zI_n)\circ^{-1}R_A^G(A)\circ \vec u\right)  \tilde{\vec y}^HV^H dz\\
        &+\frac{1}{2\pi i} \int_{\gamma_A} U \tilde{\vec x} \left(R_B^G(\bar{z}I_m)\circ^{-1}R_B^G(B^H)\circ \vec v\right)^H dz,
    \end{align*}
    and thanks to Lemma~\ref{rmk:commutativity2} the above term can be rewritten as
    \begin{align*}
        R_{h,k}&=\left(R_A^G(A)\circ \vec u\left(R_A^G(\infty)\right)^{-1}\right) \left(R_B^G(B^H)\circ \vec v\left(R_B^G(\infty)\right)^{-1}\right)^H\\
        &+\frac{1}{2\pi i} \int_{\gamma_B} \left(R_A^G(A)\circ \vec u\right)  \left({R_A^G}^H(\bar{z}I_{bk})\circ^{-1} \tilde{\vec y}\right)^HV^H dz\\
        &+\frac{1}{2\pi i} \int_{\gamma_A}U \left({R_B^G}^H(zI_{bh})\circ^{-1} \tilde{\vec x}\right) \left(R_B^G(B^H)\circ \vec v\right)^H dz.
    \end{align*}
    Finally, by Theorem~\ref{thm:generalized-cauchy} we have
    \begin{align*}
        R_{h,k}&=\left(R_A^G(A)\circ \vec u\left(R_A^G(\infty)\right)^{-1}\right) \left(R_B^G(B^H)\circ \vec v\left(R_B^G(\infty)\right)^{-1}\right)^H\\
        &+\left(R_A^G(A)\circ \vec u\right)  \left({R_A^G}^H(B_k)\circ^{-1} \vec v^{(k)}\right)^HV^H \\
        &+U \left({R_B^G}^H(A_h)\circ^{-1} \vec u^{(h)}\right) \left(R_B^G(B^H)\circ \vec v\right)^H .
    \end{align*}
    To prove \eqref{eqn:redidual-norm-sum} consider the orthogonal projectors 
    $\Pi_A=UU^H$ and $\Pi_B=VV^H$. Applying 
    Corollary~\ref{cor:char-poly} we obtain the sought identities
    \begin{align*}
        \Pi_AR_{h,k}(I_{bk}-\Pi_B)=\rho_{1,2}, &\quad (I_{bh}-\Pi_A)R_{h,k}\Pi_B=\rho_{2,1} \\
        \text{ and } \quad  (I_{bh}-\Pi_A)&R_{h,k}(I_{bk}-\Pi_B)=\rho_{2,2}. 
    \end{align*}
\end{proof}
\subsection{Pole selection}\label{sec:pole-selection} The results of
Theorem~\ref{thm:block-residual} can be used to adaptively find good poles for
the block rational Arnoldi algorithm \ref{algorithm:block-Arnoldi} for the resolution of Sylvester equations.

During this discussion we assume that one of the poles in $\boldsymbol{\xi}_A$ or $\boldsymbol{\xi}_B$ is chosen equal to infinity, hence for Remark~\ref{rmk:infinity-pole} the term $\rho_{2,2}$ in the formulation of the residual is equal to zero. With this assumption, the norm of the residual is monitored by the norms of
\begin{equation*}
    \rho_{1,2}=U({R_B^G}^H(A_h)\circ^{-1}\vec u^{(h)})(R_B^G(B^H)\circ \vec v)^H,
\end{equation*}
and
\begin{equation*}
\rho_{2,1}=\left(R_A^G(A)\circ\vec u\right)({R_A^G}^H(B_k)\circ^{-1} \vec v^{(k)})^HV^H.
\end{equation*}
Let us start by considering the norm of $\rho_{1,2}$. We have that
\begin{equation*}
    \norm{\rho_{1,2}}_F\le \norm{{R_B^G}^H(A_h)\circ^{-1}\vec u^{(h)}}_F\cdot\norm{R_B^G(B^H)\circ \vec v}_F.
\end{equation*}
By Corollary~\ref{cor:char-poly}, the vector $R_B^G(B^H)\circ \vec v$ minimizes $\norm{R(B^H)\circ \vec v}_F$ over all $R(z)\in \pol_{h}(\C^{b\times b})/Q_B(z)$ with monic numerator, for this reason, we choose the new pole by minimizing the norm of ${R_B^G}^H(A_h)\circ^{-1}\vec u^{(h)}$.

Let $\chi_B(z)=\sum_{i=0}^k\Gamma_iz^i$ be monic block characteristic polynomial of $B_k$. By the definition of the operator $\circ^{-1}$, we have 
\begin{equation*}
    \norm{{R_B^G}^H(A_h)\circ^{-1}\vec u^{(h)}}_F=\norm{(I_b\otimes \bar{Q}_B(A_h))\left(\sum_{i=0}^k\bar{\Gamma}_i\otimes A_h^i\right)^{-1}\text{vec}(\vec{v})}_2,
\end{equation*}
where $\bar{Q}_B(z)$ is the conjugate of $Q_B(z)$ and $\bar{\Gamma}_i$ denotes the conjugate of the matrix $\Gamma_i$.

Assuming for simplicity that $A_h$ is diagonalizable, i.e., $A_h=Z_hD_h{Z_h}^{-1}$ with $D_h$ diagonal matrix, we have the following bound:
\begin{equation*}
    \norm{(I_b\otimes \bar{Q}_B(A_h))\left(\sum_{i=0}^k\bar{\Gamma}_i\otimes A_h^i\right)^{-1}\text{vec}(\vec{v})}_2\le \kappa(Z_h)\norm{\vec v}_F\norm{(I_b\otimes \bar{Q}_B(D_h))\left(\sum_{i=0}^k\bar{\Gamma}_i\otimes D_h^i\right)^{-1}}_2,
\end{equation*}
where $\kappa(Z_h)$ denotes the condition number of $Z_h$. The two norm of $(I_b\otimes \bar{Q}_B(D_h))(\sum_{i=0}^k\bar{\Gamma}_i\otimes D_h^i)^{-1}$ is equal to the two norm of the matrix 
\begin{equation*}
    (\bar{Q}_B(D_h)\otimes I_b)(\sum_{i=0}^k D_h^i\otimes\bar{\Gamma}_i)^{-1}=
        \begin{bmatrix}\bar{R}_B^{-1}(\lambda_1)\\&\ddots \\ && \bar{R}_B^{-1}(\lambda_h)        
    \end{bmatrix},
\end{equation*}
where $\bar{R}_B(z)=\bar{\chi}_B(z)/\bar{Q}_B(z)=(\sum_{i=0}^hz^i\bar{\Gamma}_i)/\bar{Q}_B(z)$ and $\lambda_1,\dots, \lambda_h$ are the eigenvalues of $A_h$. In particular
\begin{equation*}
    \norm{(I_b\otimes \bar{Q}_B(D_h))\left(\sum_{i=0}^k\bar{\Gamma}_i\otimes D_h^i\right)^{-1}}_2=\max_{i=1,\dots, h}\norm{\bar{R}_B^{-1}(\lambda_i)}_2.
\end{equation*}

This shows that keeping the function $\norm{\bar{R}_B^{-1}(z)}_2$ 
small over the 
eigenvalues of $A_h$ guarantees a small norm for $\rho_{1,2}$. 
In order to obtain a condition independent of $h$, we can ask 
for $\norm{\bar{R}_B^{-1}(z)}_2$ to be small on the field of 
values of $A$, which encloses the spectra of all $A_h$. 

In the following, we describe practical methods to adaptively choose poles for $\boldsymbol{\xi}_B$. The same techniques can be used to provide poles for $\boldsymbol{\xi}_A.$

Let us assume to know the matrix $B_{k-1}$ obtained after $k-1$ steps of the block rational Arnoldi algorithm \ref{algorithm:block-Arnoldi} with poles $\boldsymbol{\xi}_{k-1}$ and that we want to choose a new pole to perform the next step of the algorithm. As we saw before the norm of $\rho_{1,2}$ after the $k$-th step can be monitored by 
\begin{equation}\label{eqn:norm_Rlambda}
    \norm{\bar{R}_B^{-1}(\lambda)}_2=|\lambda-\bar{\xi}_{k}|\cdot\norm{\bar{\chi}_k(\lambda)^{-1}\bar{Q}_{k-1}(\lambda)}_2,
\end{equation}
for $\lambda \in \W(A)$, where $Q_{k-1}(z)=\prod_{\xi\in \boldsymbol{\xi}_{k-1}, \xi\neq \infty}z-{\xi}$ and $\chi_k(z)$ is the block characteristic polynomial of $B_k$. In practice we assume that the block characteristic polynomial of $B_{k-1}$, say $\chi_{k-1}(z)$, well approximates $\chi_{k}(z)$ over $\W(A)$, hence we approximate \eqref{eqn:norm_Rlambda}, by  
\begin{equation}\label{eqn:norm_Rlambda2}
    |\lambda-\bar{\xi}_{k}|\cdot\norm{\bar{\chi}_{k-1}(\lambda)^{-1}\bar{Q}_{k-1}(\lambda)}_2.
\end{equation}
To keep \eqref{eqn:norm_Rlambda2} small over $\W(A)$ we can choose $\xi_k$ as the conjugate of
\begin{equation*}
    \arg\max_{\lambda\in \W(A)} \norm{\bar{\chi}_k(\lambda)^{-1}\bar{Q}_{k-1}(\lambda)}_2.
\end{equation*}

\begin{remark}
    If $\W(A)$ has a nonempty interior, for the maximum modulus principle it is sufficient to maximize the function over its boundary.
\end{remark}
\begin{remark}
    In the case of classical rational block Krylov, i.e., $b=1$ for the resolution of Lyapunov equations, that is $B=-A$, this result reduces to the choice of poles developed in \cite{druskin2010adaptive} for the case of $A$ symmetric and in \cite{druskin2011adaptive} for generical $A$.
\end{remark}

The numerical computation of $\xi_k$ using the definition of block
characteristic polynomial given by Theorem~\refeq{thm:blk-char-poly}, is often
inaccurate, because the condition number of the matrices $W_i$, is often large. This problem can be overcome
by developing an alternative way to compute the norm of the evaluation of block
characteristic polynomials. We leave this for future research since the result
is beyond the purpose of this work.

We now provide two methods to monitor the Euclidean norm of
$\bar{\chi}_{k-1}(\lambda)^{-1}\bar{Q}_{k-1}(\lambda)$ avoiding an explicit
computation,  noting that it equals to $1/\sigma_{\min}(\lambda)$, where
$\sigma_{\min}(\lambda)$ is the minimum singular value of
$\bar{\chi}_{k-1}(\lambda)/\bar{Q}_{k-1}(\lambda)$.

The first method is to approximate the maximizer of
${1}/{\sigma_{min}}(\lambda)$, for $\lambda \in \W(A)$, with the maximizer of
the inverse of $|\det(\bar{\chi}_{k-1}(\lambda)/\bar{Q}_{k-1}(\lambda))|$ since
the absolute value of the determinant is the product of all the singular values.
From Theorem~\ref{thm:blk-char-poly} it can be noticed that 
\begin{equation*}
\det(\bar{\chi}_{k-1}(\lambda))=\prod_{\mu\in\Lambda(B_{k-1})}(\lambda-\bar{\mu}),
\end{equation*}
hence the choice of the new pole reduces to the conjugate of
\begin{equation}\label{eqn:det}
    \arg\max_{\lambda\in \W(A)} \frac{\prod_{\xi\in \boldsymbol{\xi}_{k-1}, \xi\neq \infty}|\lambda-\bar{\xi}|^b}{\prod_{\mu\in\Lambda(B_{k-1})}|\lambda-\bar{\mu}|}.
\end{equation}

We refer to this pole selection strategy as \emph{Adaptive Determinat Minimizaztion} ({\tt{ADM}}). 

\begin{remark}
    In the case of the solution of Lyapunov equations, this choice of poles has
    already been suggested in \cite{druskin2011adaptive} as a possible
    generalization of the technique developed for non-block rational Krylov
    methods. This result produces a theoretical justification of such
    generalization and an extension to the resolution of Sylvester equations.
\end{remark}

To introduce the second method assume $B_{k-1}$ diagonalizable. In such case, if we let \begin{equation*}
    \chi_{k-1}(z)=\prod_{i=1}^{k-1}(zI_b-S_i),
\end{equation*}
as described in Theorem~\refeq{thm:blk-char-poly}, also the matrices $S_i$ are diagonalizable, hence
\begin{equation}\label{eqn:bound_det2}
    \begin{aligned}
    \norm{\bar{\chi}_{k-1}(\lambda)^{-1}\bar{Q}_{k-1}(\lambda)}_2&\le |\bar{Q}_{k-1}(\lambda)|\prod_{i=1}^{k-1}\norm{(\lambda I_b-\bar{S}_i)^{-1}}_2\\
    &\le |\bar{Q}_{k-1}(\lambda)|\prod_{i=1}^{k-1}\frac{\kappa(X_i)}{|\Lambda_{\min}(\lambda-\bar{S}_i)|},
\end{aligned}
\end{equation}
where $X_i$ is the matrix of eigenvectors of $\bar{S}_i$ and
$\Lambda_{\min}(\lambda-\bar{S}_i)$ denotes the smallest modulus eigenvalue of
$\lambda-\bar{S}_i$ for each $i$. From Theorem~\ref{thm:blk-char-poly} we see
that the matrices $S_i$ can be recovered by an arbitrary eigendecomposition of
the matrix $B_{k-1}$, in particular for a fixed $\lambda$ we can construct $S_i$
using an ordered eigendecomposition of $B_{k-1}$, where the eigenvalues
\{$\mu_i$\} of $B_{k-1}$ are ordered such that
$|\bar{\lambda}-\mu_1|\le|\bar{\lambda}-\mu_2|\le\dots\le|\bar{\lambda}-\mu_{k-1}|$.
With this construction the eigenvalues of $S_i$ are
$\mu_{(i-1)b+1},\mu_{(i-1)b+2},\dots,\mu_{ib}$ and \eqref{eqn:bound_det2} can be
rewritten as
\begin{equation*}
    \norm{\bar{\chi}_B(\lambda)^{-1}\bar{Q}_{k-1}(\lambda)}_2\le\left(\prod_{i=1}^{k-1}\kappa(X_i) \right) |\bar{Q}_B(\lambda)| \left(\prod_{i=1}^{k-1}(|\lambda-\bar{\mu}_{(i-1)b+1}|)^{-1}\right).
\end{equation*}
This suggests a new method to choose the next shift: $\xi_k$ can be taken as the
conjugate of
\begin{equation}\label{eqn:det2}
    \arg \max_{\lambda \in\W(A)}\left(\prod_{\xi\in\boldsymbol{\xi}_B, \xi\neq \infty}|\lambda-\bar{\xi}|\prod_{i=1}^{k-1}(|\lambda-\bar{\mu}_{(i-1)b+1}|)^{-1}\right),
\end{equation}
where $\mu_i$ are the eigenvalues of $B_{k-1}$ ordered as described before. 

We refer to this pole selection strategy as \emph{subsampled Adaptive Determinat Minimizaztion} ({\tt{sADM}}). 

\begin{remark}
    The main advantage of this choice of poles with respect to the previous one
    is that we have to maximize a rational function with a much smaller degree.
\end{remark}
\section{Numerical experiments} \label{sec:numerical} In this section we provide
some numerical experiment to show the convergence of the block rational Arnoldi
algorithm \ref{algorithm:block-Arnoldi} using poles determined in
Section~\ref{sec:pole-selection}: throughout the section, the algorithms that
chooses poles accordingly to \eqref{eqn:det} and \eqref{eqn:det2} are denoted by
{\tt{ADM}} and {\tt{sADM}}, respectively. The pole $\xi_0$ is always chosen
equal to infinity, and the techniques developed in
Section~\refeq{sec:swap-poles} are employed to guarantee the last pole equal to
infinity at each step. This allows computing the residual as described in
Section~\ref{sec:Krylov-Sylvester} avoiding extra computational costs. The implementation of block rational Arnoldi algorithms is based on the {\tt{rktoolbox}} for Matlab, developed in \cite{berljafa2014rational}.

The numerical simulations have been run on a Intel(R) Core(TM) i5-8250U CPU processor running Ubuntu and MATLAB R2022b.

The
experiments only involve real matrices hence, if a nonreal pole is employed, the
subsequent is chosen as its conjugate, this allows us to avoid complex matrices.
We refer the reader to \cite{ruhe1994rational2} for a more complete discussion.

In the first experiment, we compute the approximate solution of the Poisson equation 
\begin{equation*}
    \begin{cases}
        -\Delta u = f &\text{ in }\Omega\\
        u\equiv 0 &\text{ on }\partial\Omega
    \end{cases}, \qquad \Omega= [0, 1]^2.
\end{equation*}

We discretize the domain with a uniformly spaced grid with $n=4096$ points 
in each direction, and the operator $\Delta$ by finite differences, which 
yields the Lyapunov equation
\begin{equation*}
    AX+XA=F, \quad \text{ with } \quad
    A=\frac{1}{h^2}\begin{bmatrix}
        -2&1\\
        1&-2&\ddots\\
        &\ddots&\ddots&1\\
        &&1&-2
    \end{bmatrix}
\end{equation*}
where $h=\frac{1}{n-1}$ is the distance between the grid points and $F$ is the matrix obtained evaluating $f$ on the grid points. If the function $f$ is a smooth bivariate function, the matrix $F$ is numerically low-rank, that is it can be approximated by a low-rank matrix $UV^H$ where $U, V\in \C^{n\times b}$ for an appropriate $b\ll n$, see e.g. \cite[Section~2.7]{grasedyck2013literature}.

Figure~\ref{plot:diffusion} shows the behavior of the normalized residual $R_k/\norm{UV^H}_F$, for the solution of the Poisson equation with $f(x,y)=1/(1+x+y)$ with the two proposed choices of poles. In this case, the matrix $F$ has numerical rank 8. We also compared the results with the 
Extended Krylov proposed in \cite{simoncini2007new}, which is a block rational
 Krylov method that alternates a pole equal to zero and a pole equal to infinity. 
 We remark that the iterations of the extended Krylov method are usually faster than a generical block rational Krylov method since in the iterations associated with poles equal to infinity the linear systems are replaced by matrix products and the iterations associated with poles equal to zero are improved using a factorization of the matrix. Table~\ref{table:diffusion} contains times and number of iterations required to reach a relative norm of the residual less than $10^{-8}$ for the solution of discretized Poisson equation with block rational Krylov methods with different choices of poles.

\begin{figure}
	\makebox[\linewidth][c]{
		\begin{tikzpicture}
			\begin{semilogyaxis}[
				title = {},  
				xlabel = {Iteration},
				ylabel = {Normalized residual norm},
				x tick label style={/pgf/number format/.cd,%
					scaled x ticks = false,
					set thousands separator={},
					fixed},
				legend pos=north east,
				ymajorgrids=true,
                height=.45\textwidth,
				grid style=dashed]
		
				\addplot[restrict x to domain=0:21,color=blue, mark=*, mark size=2pt] table[y index = 2] {example1_ADM.dat};
				\addlegendentry{{\tt{ADM}}}

				\addplot[restrict x to domain=0:21, color=red, mark=square*, mark size=2pt] table[y index = 2] {example1_sADM.dat};
				\addlegendentry{{\tt{sADM}}}
                \addplot[restrict x to domain=0:21,color=black, mark=triangle*, mark size=2pt] table[y index = 2] {example1_ext.dat};
				\addlegendentry{{\tt{ext}}}
			\end{semilogyaxis}
	\end{tikzpicture}}
\caption{Behavior of the residual produced by solving the Poisson equation with block rational Krylov methods, with different choices of poles. }\label{plot:diffusion}
\end{figure}
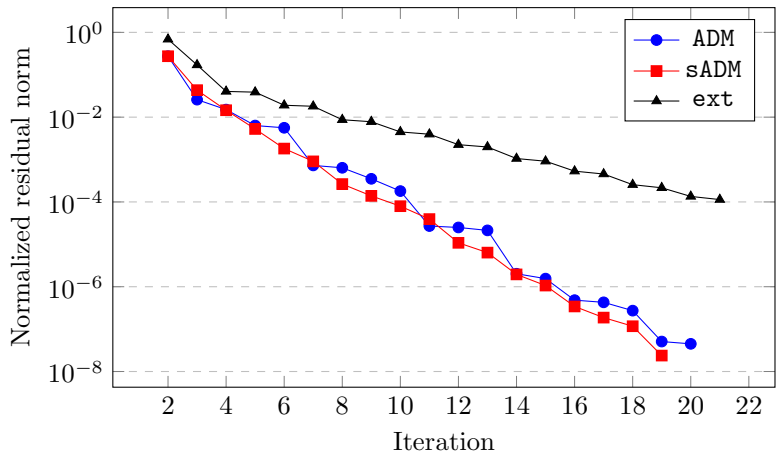

\begin{table}
    \centering
        \begin{tabular}{|c|ccc|}
            \hline
            poles &iter & residual & time (s) \\
            \hline
            ADM & $21$ & $8.82e-09$ & $0.92$ \\
            sADM & $20$ & $9.19e-09$ & $1.10$ \\
            ext & $53$ & $9.30e-09$ & $5.91$ \\
            \hline
        \end{tabular}
        \caption{Iterations and time needed to reach a relative norm of the residual less than $10^{-8}$ for the solution of discretized Poisson equation with block rational Krylov methods with different choices of poles.}\label{table:diffusion}
    \end{table}

The second experiment is the computation of an approximate solution for the convection-diffusion partial differential equation
\begin{equation*}
    \begin{cases}
        -\epsilon\Delta u +\vec w \cdot \nabla u = f &\text{ in }\Omega\\
        u\equiv 0 &\text{ on }\partial\Omega
    \end{cases}, \qquad \Omega= [0, 1]^2,
\end{equation*}
where $\epsilon\in \R_+$ is the viscosity parameter and $\vec w$ is the convection vector. Assuming $\vec w=(\Phi(x), \Psi(y))$, and discretizing the domain with a uniformly spaced grid as before, we obtain the Sylvester equation

\begin{equation*}
    (\epsilon A+\vec{\Phi}B)X+X(\epsilon A+B^H\vec{\Psi})=F
\end{equation*}
where $A$ and $F$  are defined as in the first experiment,
\begin{equation*}
    \vec {\Phi}=\begin{bmatrix}
        \Phi(h)\\
        &\Phi(2h)\\
        &&\ddots\\
        &&&\Phi((n-2)h)
    \end{bmatrix},\qquad
    \vec {\Psi}=\begin{bmatrix}
        \Psi(h)\\
        &\Psi(2h)\\
        &&\ddots\\
        &&&\Psi((n-2)h)
    \end{bmatrix}
\end{equation*}
 and 
\begin{equation*}
    B=\frac{1}{2h}\begin{bmatrix}
        0&1\\
        -1&\ddots&\ddots\\
        &\ddots&\ddots&1\\
        &&-1&0
    \end{bmatrix}
\end{equation*}
is the discretization by centered finite differences of the first order derivative in each direction.

\begin{figure}
	\makebox[\linewidth][c]{
		\begin{tikzpicture}
			\begin{semilogyaxis}[
				title = {},  
				xlabel = {Iteration},
				ylabel = {Normalized residual norm},
				x tick label style={/pgf/number format/.cd,%
					scaled x ticks = false,
					set thousands separator={},
					fixed},
				legend pos=north east,
				ymajorgrids=true,
                height=.45\textwidth,
				grid style=dashed]
		
				\addplot[restrict x to domain=0:30, color=blue, mark=*, mark size=2pt] table[y index = 2] {example2_ADM.dat};
				\addlegendentry{{\tt{ADM}}}

				\addplot[restrict x to domain=0:30, color=red, mark=square*, mark size=2pt] table[y index = 2] {example2_sADM.dat};
				\addlegendentry{{\tt{sADM}}}
                \addplot[restrict x to domain=0:30, color=black, mark=triangle*, mark size=2pt] table[y index = 2] {example2_ext.dat};
				\addlegendentry{{\tt{ext}}}
			\end{semilogyaxis}
	\end{tikzpicture}}
\caption{Behavior of the residual produced by solving the convection-diffusion equation with block rational Krylov methods, with different choices of poles. }\label{plot:convection}
\end{figure}
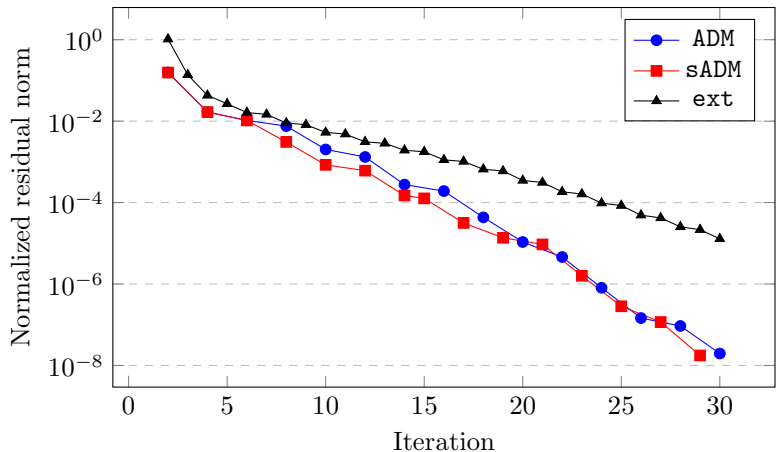

\begin{table}
    \centering
        \begin{tabular}{|c|ccc|}
            \hline
            poles &iter & residual & time (s) \\
            \hline
            ADM & $32$ & $2.18e-09$ & $2.12$ \\
            sADM & $31$ & $9.38e-09$ & $2.05$ \\
            ext & $54$ & $7.55e-09$ & $7.42$ \\
            \hline
        \end{tabular}
        \caption{Iterations and time needed to reach a relative norm of the residual less than $10^{-8}$ for the solution of discretized convection-diffusion equation with block rational Krylov methods with different choices of poles.}\label{table:convection}
    \end{table}

Figure~\ref{plot:convection} shows the behavior of the normalized residual for the solution of the convection-diffusion equation with $\epsilon=0.0083$, $f(x,y)=1/(1+x+y),$ $\vec w= (1 + \frac{(x + 1)^2}{4},\frac{1}{2}y)$ with the two proposed choices of poles and the extended Krylov method. Table~\ref{table:convection} contains times and number of iterations required to reach a relative norm of the residual less than $10^{-8}$ for the solution of discretized confection-diffusion equation with block rational Krylov methods with different choices of poles.

\section{Conclusions}

In this work we have proposed a method for solving low-rank Sylvester equations 
by means of projection onto block rational Krylov subspaces. The key 
advantage of the method with respect to state-of-the-art techniques is 
the possibility to exploit the reordering of poles to maintain the ``last''
pole of the space equal to $\infty$. This choice makes the residual of the 
large-scale equation easily computable in the projected one, without 
the need to artificially increasing the size of the subspace by 
introducing unnecessary 
poles at infinity. 

We have also reconsidered the convergence analysis for Krylov solvers 
for Sylvester equations of \cite{beckermann2011error}, extending it to 
block rational Krylov subspaces by means of the theoretical tools used 
in \cite{lund2018new} for the polynomial case. The analysis allows 
to design new strategies for adaptive pole selection, obtained by minimizing
the norm of a small $b \times b$ rational matrix, where $b$
is the block size. The minimization problem can be made simpler by replacing 
the norm with a surrogate function that is easier to evaluate.
In \cite{druskin2011adaptive} the authors 
propose a heuristic for the pole selection in block rational Krylov method, 
based on their analysis of the non-block case. Choosing the determinant 
as surrogate function yields exactly this heuristic, and it gives a solid 
theoretical justification to this approach.
Other choices, instead, yield 
completely novel strategies. One of these, called 
\texttt{sADM} in the paper, has 
comparable or better performances than the state of the art on the considered
examples. 

We expect that the results in this work will help to devise other 
pole selection strategies and convergence analysis in rational 
block Krylov methods. This will be subject to future research. 

The resulting algorithm is a 
robust solver for Sylvester equations, and the code has been made 
freely available at 
\texttt{https://github.com/numpi/rk\_adaptive\_sylvester}.

\bibliographystyle{plain}
    \bibliography{biblio_swap}

\end{document}